\ifdefined\pdfoutput
  \pdfoutput=1
\fi

\documentclass[10pt,a4paper]{article}

\usepackage[T1]{fontenc}
\usepackage[utf8]{inputenc}
\usepackage[english]{babel}
\usepackage{lmodern}
\usepackage[final,protrusion=true,expansion=true]{microtype}
\usepackage[
  a4paper,
  top=2.25cm,
  bottom=2.25cm,
  left=2.65cm,
  right=2.65cm,
  footskip=28pt,
  marginparwidth=1.75cm
]{geometry}

\usepackage{setspace}
\usepackage{amsmath,amssymb,amsfonts,amsthm,mathtools}
\usepackage{etoolbox}
\usepackage{bm}
\allowdisplaybreaks
\numberwithin{equation}{section}
\mathtoolsset{showonlyrefs=false}

\usepackage{xcolor}
\usepackage{graphicx}
\graphicspath{{figures/}{images/}}
\usepackage{booktabs}
\usepackage{array}
\usepackage{adjustbox}
\usepackage{caption}
\usepackage{subcaption}
\usepackage{float}
\usepackage[section]{placeins}

\usepackage{tikz}
\usetikzlibrary{
  arrows.meta,
  calc,
  decorations.markings,
  decorations.pathreplacing,
  positioning
}
\tikzset{
  vtx/.style={
    circle,
    fill=black,
    draw=black,
    inner sep=2pt
  },
  gluevtx/.style={
    circle,
    fill=red,
    draw=red,
    inner sep=2.15pt
  },
  ed/.style={
    blue,
    line width=1.15pt
  },
  every picture/.style={
    line cap=round,
    line join=round
  }
}

\usepackage{enumitem}
\setlist[itemize]{
  leftmargin=2.2em,
  itemsep=0.15em,
  topsep=0.35em
}
\setlist[enumerate]{
  leftmargin=2.2em,
  itemsep=0.15em,
  topsep=0.35em
}

\usepackage{authblk}

\usepackage{titlesec}

\titleformat{\section}
  {\large\bfseries}
  {\thesection.}
  {0.65em}
  {}

\titleformat{\subsection}
  {\normalsize\bfseries}
  {\thesubsection.}
  {0.65em}
  {}

\titleformat{\subsubsection}
  {\normalsize\itshape}
  {\thesubsubsection.}
  {0.65em}
  {}

\titlespacing*{\section}
  {0pt}
  {2.4ex plus .4ex minus .2ex}
  {1.1ex}

\titlespacing*{\subsection}
  {0pt}
  {1.9ex plus .3ex minus .2ex}
  {0.8ex}

\titlespacing*{\subsubsection}
  {0pt}
  {1.5ex plus .2ex minus .2ex}
  {0.6ex}

\usepackage{aliascnt}

\theoremstyle{plain}

\newtheorem{theorem}{Theorem}[section]

\newaliascnt{lemma}{theorem}
\newtheorem{lemma}[lemma]{Lemma}
\aliascntresetthe{lemma}

\newaliascnt{question}{theorem}
\newtheorem{question}[question]{Question}
\aliascntresetthe{question}

\newaliascnt{proposition}{theorem}
\newtheorem{proposition}[proposition]{Proposition}
\aliascntresetthe{proposition}

\newaliascnt{corollary}{theorem}
\newtheorem{corollary}[corollary]{Corollary}
\aliascntresetthe{corollary}

\newaliascnt{conjecture}{theorem}
\newtheorem{conjecture}[conjecture]{Conjecture}
\aliascntresetthe{conjecture}

\newaliascnt{problem}{theorem}

\aliascntresetthe{problem}

\newaliascnt{observation}{theorem}
\newtheorem{observation}[observation]{Observation}
\aliascntresetthe{observation}

\newcounter{proofgroup}

\newtheoremstyle{indentedplain}
  {\topsep}       % Space above
  {\topsep}       % Space below
  {\itshape}      % Body font
  {\parindent}    % Indentation
  {\bfseries}     % Head font
  {.}             % Punctuation after head
  {0.5em}         % Space after head
  {}              % Head specification

\theoremstyle{indentedplain}

\newtheorem{claim}{Claim}[proofgroup]

\newtheorem{case}{Case}[proofgroup]

\AtBeginEnvironment{proof}{\stepcounter{proofgroup}}

\theoremstyle{definition}

\newaliascnt{definition}{theorem}

\aliascntresetthe{definition}

\newaliascnt{example}{theorem}

\aliascntresetthe{example}

\theoremstyle{remark}

\newaliascnt{remark}{theorem}

\aliascntresetthe{remark}

\newaliascnt{fact}{theorem}

\aliascntresetthe{fact}

\makeatletter
\expandafter\patchcmd\csname\string\proof\endcsname
  {\itshape}
  {\bfseries\upshape}
  {}
  {\PackageWarning{template}{Could not make the proof heading bold}}
\makeatother

\usepackage[numbers,sort&compress]{natbib}
\definecolor{linkblue}{RGB}{0,70,150}
\definecolor{citepurple}{RGB}{82,20,160}

\usepackage[
  colorlinks=true,
  linkcolor=linkblue,
  citecolor=citepurple,
  urlcolor=black,
  bookmarksnumbered=true,
  bookmarksopen=true
]{hyperref}

\newcommand{\cZ}{\mathcal Z}
\newcommand{\cS}{\mathcal S}

\usepackage{bookmark}
\usepackage[nameinlink,noabbrev,capitalise]{cleveref}

\crefname{theorem}{Theorem}{Theorems}
\crefname{lemma}{Lemma}{Lemmas}
\crefname{proposition}{Proposition}{Propositions}
\crefname{corollary}{Corollary}{Corollaries}
\crefname{conjecture}{Conjecture}{Conjectures}
\crefname{problem}{Problem}{Problems}
\crefname{observation}{Observation}{Observations}
\crefname{claim}{Claim}{Claims}
\crefname{case}{Case}{Cases}
\crefname{definition}{Definition}{Definitions}
\crefname{example}{Example}{Examples}
\crefname{remark}{Remark}{Remarks}
\crefname{question}{Question}{Questions}
\crefname{fact}{Fact}{Facts}
\crefname{equation}{equation}{equations}
\crefname{figure}{Figure}{Figures}
\crefname{table}{Table}{Tables}
\Crefname{equation}{Equation}{Equations}
\title{\bfseries\boldmath
Every fork-free graph is perfectly weight divisible
}

\author[1]{Feng Liu\thanks{
Email: \href{mailto:liufeng0609@126.com}{liufeng0609@126.com}.}}

\author[1]{Shuang Sun\thanks{
Email: \href{mailto:chocolatesun@sjtu.edu.cn}{chocolatesun@sjtu.edu.cn}.}}

\author[1]{Yan Wang\thanks{
Email: \href{mailto:yan.w@sjtu.edu.cn}{yan.w@sjtu.edu.cn}~(Corresponding author).}}

\author[2]{Qi Wu\thanks{
Email: \href{mailto:wuqimath@163.com}{wuqimath@163.com}.}}

\author[3]{Jiasheng Zeng\thanks{
Email: \href{mailto:jasonzeng@mail.ustc.edu.cn}{jasonzeng@mail.ustc.edu.cn}.}}
\affil[1]{\scriptsize
School of Mathematical Sciences,
Shanghai Jiao Tong University,
Shanghai 200240, China
}

\affil[2]{\scriptsize  School of Mathematics and Statistics, Jiangsu Normal University, Xuzhou, Jiangsu, 221116, China}
\affil[3]{\scriptsize Department of Mathematics,
Hong Kong University of Science and Technology,
 Hong Kong}
\date{}

\begin{document}
% ============================================================

\maketitle
\thispagestyle{plain}
\vspace{-1.2em}

\begin{abstract}
A graph $G$ is \emph{perfectly weight divisible} if, for every positive
integral weight function on $V(G)$ and every induced subgraph $H$ of $G$
with at least one edge, the vertex set $V(H)$ can be partitioned into two
sets $A$ and $B$ such that $H[A]$ is perfect and the maximum weight of a
clique in $H[B]$ is smaller than the maximum weight of a clique in $H$.
Perfect divisibility and its weighted form provide a natural approach to
polynomial $\chi$-boundedness. A
\emph{fork}, also known as a \emph{chair}, is the graph obtained from a claw
by subdividing one of its edges once. In this paper, we prove that every
fork-free graph is perfectly weight divisible. As a
consequence, we confirm a conjecture of Sivaraman that every fork-free graph
is perfectly divisible.

\smallskip


\end{abstract}

\section{Introduction}
Throughout this paper, all graphs are finite and simple unless stated otherwise.
 We use standard
graph-theoretic terminology and notation; see \cite{Bondy2008,West1996}. For
a positive integer $k$, let $[k]=\{1,\ldots,k\}$. A proper $k$-coloring of a
graph $G$ is a map $\phi\colon V(G)\to [k]$ such that $\phi(u)\ne\phi(v)$
whenever $uv\in E(G)$. The chromatic number $\chi(G)$ is the least integer
$k$ for which $G$ has a proper $k$-coloring, and the clique number
$\omega(G)$ is the maximum size of a clique in $G$.
For a graph $H$, a graph $G$ is \emph{$H$-free} if it contains no induced
subgraph isomorphic to $H$. More generally, for a family $\mathcal F$ of
graphs, $G$ is \emph{$\mathcal F$-free} if it is $F$-free for every
$F\in\mathcal F$. A hereditary class $\mathcal G$ is \emph{$\chi$-bounded}
if there is a function $f$ such that $\chi(G)\leq f(\omega(G))$ for every
$G\in\mathcal G$. Such a function is called a \emph{$\chi$-binding function}
for $\mathcal G$. A central problem in this area is to determine which
classes defined by forbidden induced subgraphs are $\chi$-bounded and, when
they are, to find good $\chi$-binding functions. This line of research was
initiated by Gy\'arf\'as~\cite{Gya75}; see also the survey of Scott and
Seymour~\cite{ScottSeymourSurvey}. The most basic examples of $\chi$-bounded classes are the perfect graphs. A
graph $G$ is \emph{perfect} if $\chi(H)=\omega(H)$ for every induced
subgraph $H$ of $G$. Thus the identity function is a $\chi$-binding function
for the class of perfect graphs. A \emph{hole} is an induced cycle of length
at least four, and an \emph{antihole} is the complement of a hole. A hole or
antihole is \emph{odd} if it has an odd number of vertices. The Strong
Perfect Graph Theorem of Chudnovsky, Robertson, Seymour, and
Thomas~\cite{CRST2006} gives a complete forbidden induced subgraph
characterization of perfect graphs.

\begin{theorem}[Chudnovsky--Robertson--Seymour--Thomas~\cite{CRST2006}]
\label{thm:SPGT}
A graph is perfect if and only if it contains no odd hole and no odd
antihole.
\end{theorem}

Therefore, forbidden
induced subgraphs play an important role in coloring problems. A
construction of Erd\H{o}s~\cite{Erdos1959} shows that, for every finite
family $\mathcal F$, if the class of $\mathcal F$-free graphs is
$\chi$-bounded, then $\mathcal F$ contains a forest. This leads to the
Gy\'arf\'as--Sumner conjecture, proposed independently by
Gy\'arf\'as~\cite{Gya75} and Sumner~\cite{Sumner1981}. It asserts that the
class of $T$-free graphs is $\chi$-bounded for every tree $T$. Kierstead and Penrice~\cite{KiersteadPenrice} proved the conjecture when T has radius two, but it remains open in general.

A \emph{claw} is the complete bipartite graph $K_{1,3}$. Every line graph is
claw-free, so claw-free graphs are a natural generalization of line graphs. They
have been studied extensively; see the survey of Faudree, Flandrin, and
Ryj\'a\v{c}ek~\cite{RF1997}. Chudnovsky and Seymour~\cite{MCPS2008,ChudnovskySeymour} gave a
complete structural description of claw-free graphs.
The coloring of claw-free graphs is already quite subtle. Brause,
Randerath, Schiermeyer, and Vumar~\cite{CBBR2019} proved that even the class
of $\{3K_1,2K_2\}$-free graphs, which is a subclass of the claw-free graphs,
does not admit a linear $\chi$-binding function. On the other hand, if $G$
is claw-free, then the neighborhood of every vertex contains no independent
set of size three and no clique of size $\omega(G)$. Hence
$\chi(G)\leq\Delta(G)+1\leq R(3,\omega(G))$. Kim~\cite{JHK95} proved that
$R(3,t)=\Theta(t^2/\log t)$. It follows that
$\chi(G)=O(\omega(G)^2/\log\omega(G))$ for every claw-free graph $G$.
Chudnovsky and Seymour~\cite{MC2010} also proved that every connected
claw-free graph $G$ containing an independent set of size three satisfies
$\chi(G)\leq 2\omega(G)$.

A \emph{fork} is obtained from a claw by subdividing one edge once. It is
also called a \emph{chair}. Since every fork contains an induced claw and an
induced $P_4$, both claw-free graphs and $P_4$-free graphs are fork-free. So fork-free graphs form a natural extension of claw-free graphs.
Since a fork is a tree of radius two, the theorem of Kierstead and
Penrice~\cite{KiersteadPenrice} implies that the class of fork-free graphs
is $\chi$-bounded. Their theorem, however, does not provide a polynomial
$\chi$-binding function. This led Schiermeyer and
Randerath~\cite{SchiermeyerRanderath} to ask the following question.

\begin{question}[Schiermeyer--Randerath~\cite{SchiermeyerRanderath}]
\label{ques:fork}
Does the class of fork-free graphs have a polynomial $\chi$-binding
function?
\end{question}

Several subclasses of fork-free graphs were later shown to admit linear
$\chi$-binding functions. An \emph{antifork} is the complement of a fork.
Chudnovsky, Cook, and Seymour~\cite{ChudnovskyCookSeymour} described the
structure of graphs with no induced fork or antifork and proved that every
such graph $G$ satisfies $\chi(G)\leq 2\omega(G)$. Chudnovsky, Huang,
Karthick, and Kaufmann~\cite{ChudnovskyHuangKarthickKaufmann} proved that
every $\{\text{fork},C_4\}$-free graph $G$ satisfies
$\chi(G)\leq\lceil3\omega(G)/2\rceil$. For general fork-free
graphs, Liu, Schroeder, Wang, and Yu~\cite{LiuSchroederWangYu} answered
Question~\ref{ques:fork} affirmatively by proving that every fork-free graph
$G$ satisfies $\chi(G)\leq 7\omega(G)^2$.

A natural way to
obtain a quadratic chromatic bound is to apply induction on the clique
number. For every induced subgraph $H$ with at least one edge, one seeks a
partition $V(H)=A\mathbin{\dot\cup}B$ such that $H[A]$ is perfect and
$\omega(H[B])<\omega(H)$. The graph $H[A]$ can then be colored with at most
$\omega(H)$ colors, while the clique number of the remaining graph is
smaller. Induction gives
$\chi(G)\leq 1+2+\cdots+\omega(G)=\binom{\omega(G)+1}{2}$.
This motivates the notion of perfect divisibility, introduced by
Ho\`ang~\cite{Hoang}. A \emph{perfect division} of a graph $H$ with at least
one edge is a partition $V(H)=A\mathbin{\dot\cup}B$ such that $H[A]$ is
perfect and $\omega(H[B])<\omega(H)$. A graph $G$ is
\emph{perfectly divisible} if every induced subgraph of $G$ with at least one
edge has a perfect division. This connection between structure and coloring led Sivaraman~\cite{TKJK2022} to propose
the following conjecture.

\begin{conjecture}[Sivaraman~\cite{TKJK2022}]\label{forkconj}
Every fork-free graph is perfectly divisible.
\end{conjecture}

Chudnovsky and Sivaraman~\cite{ChudnovskySivaraman} introduced a
weighted version of perfect divisibility. Following Xu and
Zhuang~\cite{XuZhuang}, we use positive integral weight functions and
require the weighted condition for every induced subgraph. A vertex $v$ of weight $k$ can be viewed as being blown up into a clique of size $k$. This
interpretation will be used in our proof. Let
$h\colon V(G)\to\mathbb Z_{>0}$ be a positive integral weight function. For
a clique $K$ of an induced subgraph $H$ of $G$, let
$h(K)=\sum_{v\in K}h(v)$, and let $\omega_h(H)$ be the maximum value of
$h(K)$ over all cliques $K$ of $H$. An \emph{$h$-perfect division} of $H$ is
a partition $V(H)=A\mathbin{\dot\cup}B$ such that $H[A]$ is perfect and
$\omega_h(H[B])<\omega_h(H)$. A graph $G$ is
\emph{$h$-perfectly divisible} if every induced subgraph $H$ of $G$ with at
least one edge has an $h|_{V(H)}$-perfect division. It is
\emph{perfectly weight divisible} if it is $h$-perfectly divisible for every
positive integral weight function $h$ on $V(G)$. Taking $h=1$ shows that perfect weight
divisibility implies perfect divisibility.

Proving perfect divisibility for a hereditary graph class is difficult,
and only a few nontrivial classes are known to be perfectly divisible.
Most known results concern classes defined by forbidding at least two
induced subgraphs. A \emph{banner} is obtained from a $C_4$ by attaching
a pendant vertex to one of its vertices, and a \emph{bull} is obtained
from a triangle by attaching a pendant edge to each of two distinct
vertices. Ho\`ang~\cite{Hoang} proved that every
$\{\text{banner},\text{odd hole}\}$-free graph is perfectly divisible.
Chudnovsky and Sivaraman~\cite{ChudnovskySivaraman} proved that every
$\{P_5,\text{bull}\}$-free graph and every
$\{\text{odd hole},\text{bull}\}$-free graph is perfectly weight
divisible, and hence perfectly divisible. For fork-free graphs, before the present work, every positive result
toward Conjecture~\ref{forkconj} required forbidding at least one
additional induced configuration besides the fork. We recall the
configurations appearing in these results. A \emph{co-dart} is the
disjoint union of an isolated vertex and a paw, where a paw is obtained
from $K_{1,3}$ by adding an edge between two of its
leaves~\cite{TKJK2022}. A \emph{balloon} is obtained from a hole $C$ by
adding two adjacent vertices $x$ and $y$ such that $x$ has exactly two
adjacent neighbors in $V(C)$ and $y$ is anticomplete to $V(C)$. It is an \emph{odd balloon} if $C$ is
an odd hole~\cite{WuXu}. A \emph{parachute} is obtained from a hole $C$
by adding two adjacent vertices $p$ and $q$ such that $p$ is complete
to $V(C)$ and $q$ is anticomplete to $V(C)$. It is an
\emph{odd parachute} if $C$ is an odd
hole~\cite{LanLiuWuZhou}. These configurations are illustrated in
Figure~\ref{fig:graphs}.

\begin{figure}[H]
    \centering
    \tikzstyle{v}=[circle, draw, fill=black, inner sep=0pt, minimum size=5pt]

    \setlength{\tabcolsep}{0.6em}
    \renewcommand{\arraystretch}{1.0}

    \begin{tabular}{cccc}
        % claw
        \begin{tikzpicture}[scale=0.84, baseline=(current bounding box.center)]
            \useasboundingbox (-1.2,-1.35) rectangle (1.2,1.35);

            \node[v] (c) at (0,0.10) {};
            \node[v] (a) at (-0.90,-0.70) {};
            \node[v] (b) at (0.90,-0.70) {};
            \node[v] (d) at (0,1.00) {};

            \draw[blue, line width=0.9pt]
                (c)--(a)
                (c)--(b)
                (c)--(d);
        \end{tikzpicture}
        &
        % fork
        \begin{tikzpicture}[scale=0.84, baseline=(current bounding box.center)]
            \useasboundingbox (-1.2,-1.35) rectangle (1.2,1.35);

            \node[v] (c) at (0,-0.05) {};
            \node[v] (a) at (-0.90,-0.75) {};
            \node[v] (b) at (0.90,-0.75) {};
            \node[v] (s) at (0,0.65) {};
            \node[v] (d) at (0,1.20) {};

            \draw[blue, line width=0.9pt]
                (c)--(a)
                (c)--(b)
                (c)--(s)--(d);
        \end{tikzpicture}
        &
        % antifork
        \begin{tikzpicture}[scale=0.84, baseline=(current bounding box.center)]
            \useasboundingbox (-1.2,-1.35) rectangle (1.2,1.35);

            \node[v] (l) at (-0.75,0.12) {};
            \node[v] (r) at (0.75,0.12) {};
            \node[v] (t) at (0,0.94) {};
            \node[v] (b) at (0,-0.50) {};
            \node[v] (p) at (0,-1.10) {};

            \draw[blue, line width=0.9pt]
                (l)--(r)
                (l)--(t)
                (r)--(t)
                (l)--(b)
                (r)--(b)
                (b)--(p);
        \end{tikzpicture}
        &
        % co-dart
        \begin{tikzpicture}[scale=0.84, baseline=(current bounding box.center)]
            \useasboundingbox (-1.2,-1.35) rectangle (1.2,1.35);

            \node[v] (l) at (-0.70,0.35) {};
            \node[v] (r) at (0.70,0.35) {};
            \node[v] (m) at (0,-0.30) {};
            \node[v] (p) at (0,-1.00) {};
            \node[v] (i) at (0,1.10) {};

            \draw[blue, line width=0.9pt]
                (l)--(r)--(m)--(l)
                (m)--(p);
        \end{tikzpicture}
        \\[-0.20em]
        {\small\emph{claw}}
        &
        {\small\emph{fork}}
        &
        {\small\emph{antifork}}
        &
        {\small\emph{co-dart}}
        \\[1.3em]

        % banner
        \begin{tikzpicture}[scale=0.84, baseline=(current bounding box.center)]
            \useasboundingbox (-1.2,-1.35) rectangle (1.2,1.35);

            \node[v] (l) at (-0.75,0.15) {};
            \node[v] (t) at (0,0.90) {};
            \node[v] (r) at (0.75,0.15) {};
            \node[v] (b) at (0,-0.60) {};
            \node[v] (p) at (0,-1.20) {};

            \draw[blue, line width=0.9pt]
                (l)--(t)--(r)--(b)--(l)
                (b)--(p);
        \end{tikzpicture}
        &
        % bull
        \begin{tikzpicture}[scale=0.84, baseline=(current bounding box.center)]
            \useasboundingbox (-1.2,-1.35) rectangle (1.2,1.35);

            \node[v] (l) at (-0.70,0.05) {};
            \node[v] (r) at (0.70,0.05) {};
            \node[v] (b) at (0,-0.70) {};
            \node[v] (pl) at (-0.70,0.95) {};
            \node[v] (pr) at (0.70,0.95) {};

            \draw[blue, line width=0.9pt]
                (l)--(r)--(b)--(l)
                (l)--(pl)
                (r)--(pr);
        \end{tikzpicture}
        &
        % odd balloon
        \begin{tikzpicture}[scale=0.84, baseline=(current bounding box.center)]
            \useasboundingbox (-1.2,-1.35) rectangle (1.2,1.35);

            \node[v] (c1) at (-0.72,0.08) {};
            \node[v] (c2) at (0.72,0.08) {};
            \node[v] (c3) at (0.95,0.78) {};
            \node[v] (c4) at (0,1.15) {};
            \node[v] (c5) at (-0.95,0.78) {};
            \node[v] (u) at (0,-0.55) {};
            \node[v] (s) at (0,-1.20) {};

            \draw[blue, line width=0.9pt]
                (c1)--(c2)--(c3)--(c4)--(c5)--(c1);

            \draw[blue, line width=0.9pt]
                (u)--(c1)
                (u)--(c2)
                (u)--(s);
        \end{tikzpicture}
        &
        % odd parachute
        \begin{tikzpicture}[scale=0.84, baseline=(current bounding box.center)]
            \useasboundingbox (-1.2,-1.35) rectangle (1.2,1.35);

            \node[v] (h1) at (-0.95,0.20) {};
            \node[v] (h2) at (-0.58,1.00) {};
            \node[v] (h3) at (0.58,1.00) {};
            \node[v] (h4) at (0.95,0.20) {};
            \node[v] (h5) at (0,-0.92) {};
            \node[v] (u) at (0,-0.02) {};
            \node[v] (v1) at (0,0.60) {};

            \draw[blue, line width=0.9pt]
                (h1)--(h2)--(h3)--(h4)--(h5)--(h1);

            \draw[blue, line width=0.9pt]
                (u)--(h1)
                (u)--(h2)
                (u)--(h3)
                (u)--(h4)
                (u)--(h5)
                (u)--(v1);
        \end{tikzpicture}
        \\[-0.20em]
        {\small\emph{banner}}
        &
        {\small\emph{bull}}
        &
        {\small\emph{odd balloon}}
        &
        {\small\emph{odd parachute}}
    \end{tabular}

    \caption{Illustrations of the configurations.}
    \label{fig:graphs}
\end{figure}
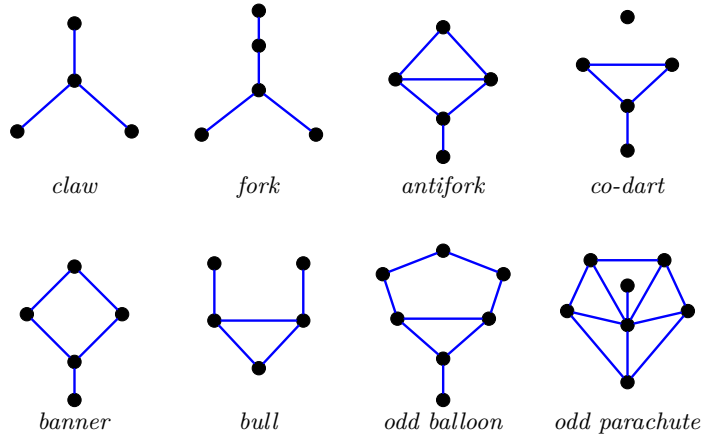

Karthick, Kaufmann, and Sivaraman~\cite{TKJK2022} proved
Conjecture~\ref{forkconj} for $\{\text{fork},F\}$-free graphs when
$F\in\{P_6,\text{co-dart},\text{bull}\}$. Wu and Xu~\cite{WuXu}
proved that every $\{\text{fork},\text{odd balloon}\}$-free graph is
perfectly divisible. As observed by Xu and Zhuang~\cite{XuZhuang}, their
argument in fact proves perfect weight divisibility. Xu and
Zhuang~\cite{XuZhuang} also proved that every
$\{\text{fork},P_7\}$-free graph and every
$\{\text{fork},P_6\cup K_1\}$-free graph is perfectly weight divisible.
More recently, Lan, Liu, Wu, and Zhou~\cite{LanLiuWuZhou} proved that every
$\{\text{fork},\text{odd parachute}\}$-free graph is either perfectly
divisible or has a \emph{trisimplicial vertex}, that is, a vertex whose
neighborhood is the union of three cliques. In this paper, we prove the following.

\begin{theorem}\label{thm:main}
Every fork-free graph is perfectly weight divisible.
\end{theorem}

Since perfect weight divisibility implies perfect divisibility,
Theorem~\ref{thm:main} confirms Conjecture~\ref{forkconj}. 
As a corollary, we prove the following upper bound on the chromatic number of fork-free graphs.

\begin{corollary}\label{cor:fork-coloring}
Every fork-free graph $G$ satisfies
$\chi(G)\leq\binom{\omega(G)+1}{2}$.
\end{corollary}

\begin{proof}
By Theorem~\ref{thm:main}, every fork-free graph is perfectly divisible.
The result follows by induction on $\omega(G)$. Indeed, a perfect division
$V(G)=A\mathbin{\dot\cup}B$ satisfies
$\chi(G[A])\leq\omega(G)$ and
$\omega(G[B])<\omega(G)$. Applying induction to $G[B]$ gives
$\chi(G)\leq\omega(G)+\binom{\omega(G)}{2}
=\binom{\omega(G)+1}{2}$.
\end{proof}
This
improves the bound $\chi(G)\leq7\omega(G)^2$ proved by Liu, Schroeder,
Wang, and Yu~\cite{LiuSchroederWangYu}.

The paper is organized as follows.
Section~\ref{sec:preliminaries} gives the reduction to claw-free graphs,
introduces perfect transversals, and establishes the properties of a minimal
counterexample.
Section~\ref{sec:main-claw-free} introduces the structural terminology,
states the required theorems of Chudnovsky and Seymour, and derives the
graph-level reduction used in the proof.
Section~\ref{sec:basic-classes} treats the three-clique outcome and
thickenings of $\cS_1$, $\cS_3$, and $\cS_7$.
Section~\ref{sec:stripes} constructs suitable transversals for the two-marker
stripe classes $\cZ_1,\ldots,\cZ_5$, and Section~\ref{sec:composition}
combines these local transversals.
Section~\ref{sec:main-results} proves the claw-free result and then
Theorem~\ref{thm:main}. Finally, we conclude in Section~\ref{concluding}.

\section{Perfect transversals and minimal counterexamples}
\label{sec:preliminaries}
We begin by fixing some notation that will be used throughout the proofs.
For $X\subseteq V(G)$, we write $G[X]$ for the subgraph of $G$ induced by
$X$. For $v\in V(G)$ and $Y\subseteq V(G)$, let $N_G(v)$ and $N_G[v]$
denote the open and closed neighborhoods of $v$, respectively, and put
$N_Y(v)=N_G(v)\cap Y$ and $M_G(v)=V(G)\setminus N_G[v]$. If the graph is
clear from the context, we omit the subscript. For $u,v\in V(G)$, we simply write $u\sim v$ if $uv \in E(G)$, and write $u\nsim v$ if $uv \notin E(G)$. A graph is a \emph{minimal non-perfectly weight divisible graph} if it is not
perfectly weight divisible but every proper induced subgraph is perfectly
weight divisible. Xu and Zhuang~\cite{XuZhuang} proved the following
reduction.

\begin{theorem}[Xu--Zhuang~\cite{XuZhuang}]
\label{thm:xu-zhuang}
Every minimal non-perfectly weight divisible fork-free graph is claw-free.
\end{theorem}

By Theorem~\ref{thm:xu-zhuang}, it suffices to prove that every
claw-free graph is perfectly weight divisible. This will be done in the
rest of the paper.

Let $G$ be a graph with at least one edge, and let $\Omega_G$ denote the
family of all maximum cliques of $G$. A set $S\subseteq V(G)$ is an
\emph{$\Omega_G$-transversal} if $S\cap K\neq\emptyset$ for every
$K\in\Omega_G$. If, in addition, $G[S]$ is perfect, then $S$ is called a
\emph{perfect $\Omega_G$-transversal}. For the weighted version, let
$h\colon V(G)\to\mathbb Z_{>0}$ be a positive integral weight function,
and let $\Omega_{G,h}$ denote the family of all cliques $K$ of $G$ with
$h(K)=\omega_h(G)$. A set $S\subseteq V(G)$ is an
\emph{$\Omega_{G,h}$-transversal} if $S\cap K\neq\emptyset$ for every
$K\in\Omega_{G,h}$. If, in addition, $G[S]$ is perfect, then $S$ is
called a \emph{perfect $\Omega_{G,h}$-transversal}.

\begin{observation}\label{obs:weighted-transversal}
Let $G$ be a graph with at least one edge, and let
$h\colon V(G)\to\mathbb Z_{>0}$. Then $G$ has an $h$-perfect division if
and only if it has a perfect $\Omega_{G,h}$-transversal. In particular,
$G$ has a perfect division if and only if it has a perfect
$\Omega_G$-transversal.
\end{observation}

\begin{proof}
For a partition $V(G)=A\mathbin{\dot\cup}B$, the inequality
$\omega_h(G[B])<\omega_h(G)$ holds if and only if $A$ meets every member
of $\Omega_{G,h}$. 
\end{proof}

\subsection{Properties of minimal counterexamples}
Throughout this subsection, let $G$ be a minimal non-perfectly weight
divisible claw-free graph. Since $G$ is not perfectly weight divisible, there is a positive integral weight function $h$ on
$V(G)$ such that $G$ has no $h$-perfect division. Fix such a function $h$.
By Observation~\ref{obs:weighted-transversal}, $G$ has no perfect
$\Omega_{G,h}$-transversal. For every induced subgraph $H$ of $G$, the
notation $\Omega_{H,h}$ refers to the family of maximum-weight cliques of
$H$ with respect to $h|_{V(H)}$. Let $\mathcal C_G$ denote the set of
components of $G$.

A vertex is \emph{simplicial} if its closed neighborhood is a clique. A
\emph{clique cutset} of $G$ is a clique $X\subseteq V(G)$ such that $G-X$
is disconnected. Hu, Xu, and Zhuang~\cite[Theorems~3.1\textup{(1)}
and~1.4]{HXZ2026} proved in the unweighted setting that a minimal
non-perfectly divisible claw-free graph has no simplicial vertex and no
clique cutset. We prove the following weighted analogues.

\begin{lemma}\label{lem:minimal}
$G$ is connected and has no simplicial vertex.
\end{lemma}

\begin{proof}
Suppose first that $G$ is disconnected. For every
$D\in\mathcal C_G$ containing an edge, let $S_D$ be a perfect
$\Omega_{D,h}$-transversal; such a set exists because $D$ is a proper
induced subgraph of $G$. If $D$ is an isolated vertex, put
$S_D=V(D)$. Then $S=\bigcup_{D\in\mathcal C_G}S_D$ induces a disjoint
union of perfect graphs and hence is perfect. Every member of
$\Omega_{G,h}$ is a maximum-weight clique of one component, so it meets
$S$. Thus $S$ is a perfect $\Omega_{G,h}$-transversal, a contradiction.

Suppose next that $v$ is a simplicial vertex of $G$. If $G-v$ contains an
edge, let $S$ be a perfect $\Omega_{G-v,h}$-transversal; otherwise, put
$S=V(G-v)$. Since $v$ is simplicial, $G[S\cup\{v\}]$ is perfect. Indeed,
let $J$ be an induced subgraph of $G[S\cup\{v\}]$ containing $v$, and put
$k=\omega(J)$. By perfection of $J-v$, it has a proper $k$-coloring.
The set $N_J(v)$ is a clique and $N_J[v]$ is also a clique, so
$|N_J(v)|\leq k-1$; hence some color is absent from $N_J(v)$ and can be
assigned to $v$. Induced subgraphs not containing $v$ are perfect as
induced subgraphs of $G[S]$. Moreover,
every member of $\Omega_{G,h}$ either contains $v$ or belongs to
$\Omega_{G-v,h}$. Hence $S\cup\{v\}$ is a perfect
$\Omega_{G,h}$-transversal,  a contradiction.
\end{proof}

By Lemma~\ref{lem:minimal}, $G$ is connected. We next show that $G$ has no
clique cutset.

\begin{lemma}\label{lem:cutset}
$G$ has no clique cutset.
\end{lemma}

\begin{proof}
Suppose for a contradiction that $G$ has a clique cutset. Choose a clique cutset $X$ of minimum cardinality. There is a partition $(V_1,X,V_2)$ of $V(G)$ such that $V_1,V_2\neq\emptyset$ and $V_1$ is anticomplete to $V_2$. For $i\in[2]$, put $G_i=G[X\cup V_i]$, and let $r=\omega_h(G)$. Since $G$ is connected, each $G_i$ contains an edge; and since the other side is nonempty, each $G_i$ is a proper induced subgraph of $G$. Thus an $h$-perfect division of either $G_i$ exists whenever it is invoked below.

\begin{claim}\label{claim:clique-r}
$\omega_h(G_1)=\omega_h(G_2)=r$.
\end{claim} \vspace{-0.6em}

Suppose not. By symmetry, assume that $\omega_h(G_1)<r$. Since every clique of $G$ is contained in $G_1$ or $G_2$, we have $\omega_h(G_2)=r$. Since $G_2$ is a proper induced subgraph of $G$, it has an $h$-perfect division $(A_2,B_2)$. Thus $G[A_2]$ is perfect and $\omega_h(G_2[B_2])<r$. Every clique of $G-A_2$ is contained in $G_1$ or $G_2[B_2]$. Hence $\omega_h(G-A_2)\leq\max\{\omega_h(G_1),\omega_h(G_2[B_2])\}<r$. Therefore $(A_2,V(G)\setminus A_2)$ is an $h$-perfect division of $G$, a contradiction. This proves Claim~\ref{claim:clique-r}.\hfill\ensuremath{\blacksquare}

For each $i\in[2]$, let $(A_i,B_i)$ be an $h$-perfect division of $G_i$.

\begin{claim}\label{claim:clique-nonempty}
$(A_1\cap B_2)\cup(A_2\cap B_1)\neq\emptyset$.
\end{claim} \vspace{-0.6em}

Suppose otherwise. Then $A_1\cap B_2=A_2\cap B_1=\emptyset$, and hence $A_1\cap X=A_2\cap X$ and $B_1\cap X=B_2\cap X$. Therefore $(A_1\cup A_2,B_1\cup B_2)$ is a partition of $V(G)$.

We first show that $G[A_1\cup A_2]$ is perfect. Let $H$ be an induced subgraph of $G[A_1\cup A_2]$, and put $k=\omega(H)$. Since $G_i[A_i]$ is perfect, each $H[A_i]$ has a proper $k$-coloring. Their intersection is a clique contained in $X$, so the colors may be relabeled so that the two colorings agree on the intersection. Since $A_1\setminus X$ is anticomplete to $A_2\setminus X$, the two colorings combine to give a proper $k$-coloring of $H$. Thus $G[A_1\cup A_2]$ is perfect.

Every clique of $G[B_1\cup B_2]$ is contained in $G_1[B_1]$ or $G_2[B_2]$. Therefore $\omega_h(G[B_1\cup B_2])=\max\{\omega_h(G_1[B_1]),\omega_h(G_2[B_2])\}<r$. Hence $(A_1\cup A_2,B_1\cup B_2)$ is an $h$-perfect division of $G$, a contradiction. This proves Claim~\ref{claim:clique-nonempty}.\hfill\ensuremath{\blacksquare}

\begin{claim}\label{claim:clique-imperfect}
For each $i\in [2]$, the graph $G_i$ is not perfect.
\end{claim} \vspace{-0.6em}

Suppose not. By symmetry, assume that $G_1$ is perfect. Let $(A_2,B_2)$ be an $h$-perfect division of $G_2$. Put $B_1=B_2\cap X$ and $A_1=V(G_1)\setminus B_1$. Since $G_1$ is perfect, so is $G_1[A_1]$. Moreover, $\omega_h(G_1[B_1])\leq\omega_h(G_2[B_2])<r=\omega_h(G_1)$. Thus $(A_1,B_1)$ is an $h$-perfect division of $G_1$.

By the definitions of $A_1$ and $B_1$, we have $A_1\cap X=A_2\cap X$ and $B_1\cap X=B_2\cap X$. Since $V(G_1)\cap V(G_2)=X$, it follows that $A_1\cap B_2=A_2\cap B_1=\emptyset$, contrary to Claim~\ref{claim:clique-nonempty}. This proves Claim~\ref{claim:clique-imperfect}.\hfill\ensuremath{\blacksquare}

\begin{claim}\label{claim:clique-neighbor}
For every $x\in X$ and $i\in[2]$, the set $N_{V_i}(x)$ is a nonempty clique.
\end{claim} \vspace{-0.6em}

Fix $x\in X$. Suppose that $x$ has no neighbor in $V_i$ for some $i\in[2]$. If $X=\{x\}$, then $V_i$ is anticomplete to $V(G)\setminus V_i$, contrary to the connectedness of $G$. If $|X|\geq2$, then $X\setminus\{x\}$ is a clique cutset smaller than $X$, again a contradiction. Hence $N_{V_i}(x)\neq\emptyset$ for every $i\in[2]$.

Suppose that $N_{V_i}(x)$ is not a clique for some $i\in[2]$. By symmetry, assume that $N_{V_1}(x)$ is not a clique. Choose nonadjacent vertices $a,b\in N_{V_1}(x)$ and a vertex $c\in N_{V_2}(x)$. Since $V_1$ is anticomplete to $V_2$, the set $\{x,a,b,c\}$ induces a claw centered at $x$, a contradiction. This proves Claim~\ref{claim:clique-neighbor}.\hfill\ensuremath{\blacksquare}

Among all pairs of $h$-perfect divisions $(A_i,B_i)$ of $G_i$, for $i\in[2]$, choose one minimizing $d(A_1,B_1;A_2,B_2)=|(A_1\cap B_2)\cup(A_2\cap B_1)|$. By Claim~\ref{claim:clique-nonempty}, this number is positive.

\begin{claim}\label{claim:select}
Let $i,j\in[2]$ be distinct and let $x\in A_i\cap B_j$. Then
$\omega_h(G[B_i\cup\{x\}])=r$ and $G[A_j\cup\{x\}]$ is not perfect.
\end{claim} \vspace{-0.6em}

By symmetry, let $x\in A_1\cap B_2$. Since
$V(G_1)\cap V(G_2)=X$, we have $x\in X$. Suppose first that
$\omega_h(G[B_1\cup\{x\}])<r$. Since
$G[A_1\setminus\{x\}]$ is perfect,
$(A_1\setminus\{x\},B_1\cup\{x\})$ is an $h$-perfect division of
$G_1$. Only the membership of $x$ changes, from the first side of the
first division to its second side. Thus $x$ is removed from the set counted by $d$, and no new
vertex enters this set. Therefore
$d(A_1,B_1;A_2,B_2)$ decreases by one. This contradicts the choice of
the two divisions. Hence $\omega_h(G[B_1\cup\{x\}])=r$.

Suppose next that $G[A_2\cup\{x\}]$ is perfect. Since
$\omega_h(G[B_2\setminus\{x\}])\leq\omega_h(G[B_2])<r$, the pair
$(A_2\cup\{x\},B_2\setminus\{x\})$ is an $h$-perfect division of
$G_2$. Again only the membership of $x$ changes, now aligning the two
divisions on $x$, so the value of $d(A_1,B_1;A_2,B_2)$ decreases by
one. This is a contradiction. Therefore $G[A_2\cup\{x\}]$ is not
perfect. This proves Claim~\ref{claim:select}.\hfill\ensuremath{\blacksquare}

\medskip
By symmetry, assume that $A_1\cap B_2\neq\emptyset$, and choose $x_0\in A_1\cap B_2$. By Claim~\ref{claim:select}, $\omega_h(G[B_1\cup\{x_0\}])=r$ and $G[A_2\cup\{x_0\}]$ is not perfect. Since $G[A_2]$ is perfect, the Strong Perfect Graph Theorem implies that $G[A_2\cup\{x_0\}]$ contains an odd hole or an odd antihole $C$ with $x_0\in V(C)$. Since a $5$-antihole is also a $5$-hole, it is enough to consider the following two cases.

\begin{case}
$C$ is an odd hole.
\end{case}

Write $C=x_0v_1v_2\ldots v_nx_0$, where $n\geq4$ is even. The vertices $v_1$ and $v_n$ are nonadjacent neighbors of $x_0$. They cannot both belong to $X$, since $X$ is a clique, and they cannot both belong to $V_2$, since $N_{V_2}(x_0)$ is a clique by Claim~\ref{claim:clique-neighbor}. By reversing the cyclic order if necessary, assume that $v_1\in X$ and $v_n\in V_2$. Since $v_2,\ldots,v_{n-1}$ are nonadjacent to $x_0$, none of them belongs to $X$. Therefore $v_1\in X$ and $v_2,\ldots,v_n\in V_2$.

Put $Q=N_{V_1}(x_0)$. By Claim~\ref{claim:clique-neighbor}, $Q$ is a nonempty clique.

\begin{claim}\label{claim:common-neighborhood}
$N_{V_1}(x)=Q$ for every $x\in X$.
\end{claim} \vspace{-0.6em}

We first prove that $N_{V_1}(v_1)=Q$. If some $y\in Q$ is nonadjacent to $v_1$, then $\{x_0,y,v_1,v_n\}$ induces a claw centered at $x_0$. Hence $Q\subseteq N_{V_1}(v_1)$. Conversely, if some $y\in N_{V_1}(v_1)$ is nonadjacent to $x_0$, then $\{v_1,y,x_0,v_2\}$ induces a claw centered at $v_1$. Thus $N_{V_1}(v_1)=Q$.

Let $x\in X\setminus\{x_0,v_1\}$. We first show that $N_{V_1}(x)\subseteq Q$. Suppose otherwise, and choose $y\in N_{V_1}(x)\setminus Q$. By Claim~\ref{claim:clique-neighbor}, choose $z\in N_{V_2}(x)$ and $q\in Q$. Since $y\notin Q=N_{V_1}(x_0)$ and $N_{V_1}(v_1)=Q$, the vertex $y$ is nonadjacent to both $x_0$ and $v_1$.

If $z=v_2$, then $\{x,y,v_2,x_0\}$ induces a claw centered at $x$. Similarly, if $z=v_n$, then $\{x,y,v_n,v_1\}$ induces a claw centered at $x$. Hence $z\notin\{v_2,v_n\}$.

Since $V_1$ is anticomplete to $V_2$, we have $y\nsim z$. The sets $\{x,y,z,x_0\}$ and $\{x,y,z,v_1\}$ cannot induce claws, so $z$ is adjacent to both $x_0$ and $v_1$. Since $q$ is anticomplete to $V_2$, the sets $\{x_0,q,z,v_n\}$ and $\{v_1,q,z,v_2\}$ cannot induce claws. Therefore $z$ is adjacent to both $v_n$ and $v_2$.

Since $v_2\nsim v_n$, the set $\{z,x,v_2,v_n\}$ cannot induce a claw centered at $z$. Thus $x$ is adjacent to $v_2$ or $v_n$. In the first case, $\{x,y,v_2,x_0\}$ induces a claw centered at $x$; in the second case, $\{x,y,v_n,v_1\}$ induces a claw centered at $x$. Both are impossible. Hence $N_{V_1}(x)\subseteq Q$.

It remains to prove that $Q\subseteq N_{V_1}(x)$. Suppose that some $q'\in Q$ is nonadjacent to $x$. Choose $y\in N_{V_1}(x)$, which belongs to $Q$ by the preceding paragraph. Since neither $\{x_0,q',v_n,x\}$ nor $\{v_1,q',v_2,x\}$ induces a claw, the vertex $x$ is adjacent to both $v_n$ and $v_2$. It follows that $\{x,y,v_2,v_n\}$ induces a claw centered at $x$, a contradiction. Therefore $Q\subseteq N_{V_1}(x)$. This proves Claim~\ref{claim:common-neighborhood}.\hfill\ensuremath{\blacksquare}

\medskip
By Claim~\ref{claim:common-neighborhood}, $X\cup Q$ is a clique and every vertex of $X$ is simplicial in $G_1$. By Claim~\ref{claim:select}, $\omega_h(G[B_1\cup\{x_0\}])=r$. Hence there is a clique $K\subseteq B_1\cup\{x_0\}$ with $h(K)=r$. Since $\omega_h(G[B_1])<r$, we have $x_0\in K$. Therefore $K\subseteq N_{G_1}[x_0]=X\cup Q$.

Since $X\cup Q$ is a clique of $G_1$, we have $r=h(K)\leq h(X\cup Q)\leq\omega_h(G_1)=r$. Thus $h(K)=h(X\cup Q)$. All weights are positive, so $K=X\cup Q$. In particular, $X\setminus\{x_0\}\subseteq B_1$.

Now $v_1\in X\setminus\{x_0\}$, so $v_1\in B_1$. Since $V(C)\subseteq A_2\cup\{x_0\}$, we also have $v_1\in A_2$. Thus $v_1\in A_2\cap B_1$, and Claim~\ref{claim:select} implies that $G[A_1\cup\{v_1\}]$ is not perfect.

On the other hand, $v_1$ is simplicial in $G_1$. Since adding a simplicial vertex to a perfect graph preserves perfection, $G[A_1\cup\{v_1\}]$ is perfect, a contradiction.

\begin{case}
$C$ is an odd antihole of length at least seven.
\end{case}

Write $V(C)=\{u_0,u_1,\ldots,u_k\}$, where $u_0=x_0$, $k\geq6$ is even, and $u_0u_1\ldots u_ku_0$ is the corresponding cycle in $\overline{G[V(C)]}$. $\{u_1,u_k\}$ is anticomplete to $u_0$. Since $V(C)\subseteq X\cup V_2$, neither belongs to $X$. Thus $u_1,u_k\in V_2$.

Every vertex in $\{u_2,\ldots,u_{k-1}\}$ belongs to $X\cup V_2$ and is adjacent to $u_0$. Two consecutive vertices in this set cannot both belong to $X$, since they are nonadjacent. They cannot both belong to $V_2$, since $N_{V_2}(u_0)$ is a clique. Their memberships therefore alternate. By reversing the order if necessary, assume that $u_2,u_4,\ldots,u_{k-2}\in X$ and $u_1,u_3,\ldots,u_{k-1},u_k\in V_2$. Put $E=\{u_2,u_4,\ldots,u_{k-2}\}$.

We first show that every vertex outside $C$ has a neighbor in $C$. Suppose otherwise, and put $M=\{x\in V(G)\setminus V(C):N_G(x)\cap V(C)=\emptyset\}$. Since $G$ is connected, there are adjacent vertices $a\in M$ and $u\notin M$. The vertex $a$ is anticomplete to $C$, while $u\notin V(C)$ has a neighbor in $C$.

Relabel the vertices of $C$ as $c_0,c_1,\ldots,c_k$ so that $c_0c_1\ldots c_kc_0$ is the corresponding cycle in $\overline{G[V(C)]}$ and $u\sim c_0$. For $i\in\{0,\ldots,k\}$, let $s_i=1$ if $u\sim c_i$, and let $s_i=0$ otherwise.

The set $N_C(u)$ is a clique. Otherwise, two nonadjacent vertices $c_i,c_j\in N_C(u)$, together with $a$, form a claw centered at $u$. Hence no two consecutive entries in the cyclic sequence $s_0,s_1,\ldots,s_k$ are both $1$. In particular, $s_1=s_k=0$.

For every $i\in\{2,\ldots,k-2\}$, the entries $s_i$ and $s_{i+1}$ cannot both be $0$, since otherwise $\{c_0,u,c_i,c_{i+1}\}$ induces a claw centered at $c_0$. Therefore $s_2,s_3,\ldots,s_{k-1}$ alternate.

If $s_2=1$, then $s_{k-1}=0$, and $\{c_2,u,c_{k-1},c_k\}$ induces a claw centered at $c_2$. If $s_2=0$, then $s_{k-1}=1$, and $\{c_{k-1},u,c_1,c_2\}$ induces a claw centered at $c_{k-1}$. Both cases give a contradiction. Thus every vertex outside $C$ has a neighbor in $C$.

\begin{claim}\label{claim:antihole-complete}
Every vertex of $V_1$ is complete to $E$.
\end{claim} \vspace{-0.6em}

Let $y\in V_1$. By the preceding paragraph, $y$ has a neighbor in $C$. Since $V_1$ is anticomplete to $V_2$, every neighbor of $y$ in $C$ belongs to $E\cup\{u_0\}$. If $y$ is anticomplete to $E$, then $y\sim u_0$, and $\{u_0,y,u_2,u_3\}$ induces a claw centered at $u_0$. Hence $y\sim u_j$ for some $u_j\in E$.

Suppose that $j\leq k-4$ and $y\nsim u_{j+2}$. Since $u_{j+3}\in V_2$, we have $y\nsim u_{j+3}$. Therefore $\{u_j,y,u_{j+2},u_{j+3}\}$ induces a claw centered at $u_j$, a contradiction. Hence $y\sim u_{j+2}$.

Similarly, if $j\geq4$ and $y\nsim u_{j-2}$, then $y\nsim u_{j-3}$ because $u_{j-3}\in V_2$. Thus $\{u_j,y,u_{j-2},u_{j-3}\}$ induces a claw centered at $u_j$, again a contradiction. Hence $y\sim u_{j-2}$. Repeating these arguments shows that $y$ is complete to $E$. This proves Claim~\ref{claim:antihole-complete}.\hfill\ensuremath{\blacksquare}

\medskip
By Claim~\ref{claim:antihole-complete}, $u_2$ is complete to $V_1$. Since $N_{V_1}(u_2)$ is a clique by Claim~\ref{claim:clique-neighbor}, $V_1$ is a clique. Thus $V(G_1)=X\cup V_1$ is the union of two cliques. Hence $G_1$ is cobipartite and therefore perfect, contrary to Claim~\ref{claim:clique-imperfect}. This completes the proof of Lemma~\ref{lem:cutset}.
\end{proof}

We shall also use the following elementary fact about joining two
perfect graphs through a pair of cliques.

\begin{lemma}\label{lem:clique-attachment}
Let $G_1$ and $G_2$ be vertex-disjoint perfect graphs, and let $Q_i$
be a clique (possibly empty) of $G_i$ for each $i\in\{1,2\}$. Let $G$ be obtained from
the disjoint union of $G_1$ and $G_2$ by making $Q_1$ complete to
$Q_2$. Then $G$ is perfect.
\end{lemma}

\begin{proof}
Let $F$ be an induced subgraph of $G$, and put $k=\omega(F)$. For
$i\in\{1,2\}$, let $F_i=F[V(F)\cap V(G_i)]$ and
$D_i=V(F)\cap Q_i$. Since $G_i$ is perfect, $F_i$ has a proper
$k$-coloring. Moreover, $D_1\cup D_2$ is a clique of $F$, so
$|D_1|+|D_2|\leq k$. Relabel the colors in the two colorings so that $D_1$ and $D_2$ use
disjoint sets of colors. Since the only edges between $F_1$ and $F_2$
join $D_1$ to $D_2$, the two colorings together give a proper
$k$-coloring of $F$. Thus $\chi(F)\leq k=\omega(F)$, and hence $G$ is
perfect.
\end{proof}

\section{The structure theorem for claw-free graphs}
\label{sec:main-claw-free}
We first introduce the terminology used in the structural results of
Chudnovsky and Seymour~\cite{ChudnovskySeymour}. We then state the
 results, derive the form needed for graphs with no simplicial vertex
or clique cutset, and identify exactly which basic and stripe classes still remain to handle in our setting.

\subsection{Terminology and notation}
A \emph{trigraph} $T$ consists of a finite vertex set $V(T)$ and a map
$\theta_T\colon\binom{V(T)}{2}\to\{-1,0,1\}$ such that the pairs assigned
value $0$ form a matching. Distinct vertices $u,v\in V(T)$ are
\emph{strongly adjacent}, \emph{semiadjacent}, or
\emph{strongly antiadjacent} according as $\theta_T(\{u,v\})$ is $1$, $0$,
or $-1$. They are \emph{adjacent} if they are strongly adjacent or
semiadjacent, and \emph{antiadjacent} if they are strongly antiadjacent or
semiadjacent. Thus a semiadjacent pair is both adjacent and
antiadjacent~\cite[Section~2]{ChudnovskySeymour}.

For disjoint sets $A,B\subseteq V(T)$, we say that $A$ is \emph{complete}
to $B$ if every vertex of $A$ is adjacent to every vertex of $B$, and
\emph{strongly complete} to $B$ if all these pairs are strongly adjacent.
The terms \emph{anticomplete} and \emph{strongly anticomplete} are defined
similarly. A \emph{clique} is a set of pairwise adjacent vertices, and a
\emph{strong clique} is a set of pairwise strongly adjacent vertices. A set
is \emph{independent} if its vertices are pairwise antiadjacent, and
\emph{strongly stable} if its vertices are pairwise strongly antiadjacent.
A \emph{claw} consists of vertices $x,a,b,c$ such that $x$ is adjacent to
$a,b,c$, while $a,b,c$ are pairwise antiadjacent. A trigraph is
\emph{claw-free} if it contains no claw.

For $X\subseteq V(T)$, let $T[X]$ denote the induced subtrigraph on $X$, and
write $T-X=T[V(T)\setminus X]$. We view an ordinary graph as a trigraph in
which every edge is a strong adjacency and every nonedge is a strong
antiadjacency. Following Chudnovsky and Seymour~\cite[Section~2]{ChudnovskySeymour}, a
trigraph $H$ is a \emph{thickening} of a trigraph $T$ if $V(H)$ has a
partition $(X_v:v\in V(T))$ into nonempty strong cliques such that, for
distinct $u,v\in V(T)$,

\begin{enumerate}[label=\textup{$(\roman*)$}]
\item $X_u$ is strongly complete to $X_v$ if $u$ and $v$ are strongly
adjacent in $T$;

\item $X_u$ is strongly anticomplete to $X_v$ if $u$ and $v$ are strongly
antiadjacent in $T$;

\item $X_u$ is neither strongly complete nor strongly anticomplete to $X_v$
if $u$ and $v$ are semiadjacent in $T$.
\end{enumerate}
The sets $X_v$ are the \emph{bags} of the thickening, and $T$ is its
\emph{base trigraph}.

For $z\in V(T)$, let $N_T(z)$ be the set of vertices adjacent to $z$. The
vertex $z$ is \emph{simplicial} if $N_T(z)\cup\{z\}$ is a strong clique.
Following Chudnovsky and Seymour~\cite[Section~7]{ChudnovskySeymour}, a
\emph{stripe} is a pair $(T,Z)$, where $T$ is a claw-free trigraph and $Z$
is a strongly independent set of simplicial vertices such that no vertex of $T$
has two neighbors in $Z$. The vertices of $Z$ are the \emph{markers}.

Let $(T^0,Z^0)$ be a stripe, and let $G$ be a thickening of $T^0$ with bags
$(X_v:v\in V(T^0))$ such that $|X_z|=1$ for every $z\in Z^0$. Identifying
each marker with the unique vertex of its bag, we call $(G,Z^0)$ a
\emph{thickening} of $(T^0,Z^0)$. If $Z^0=\{z_1,z_2\}$, then
$S=G-Z^0$ is the \emph{piece} of the thickening, and
$A_i=N_G(z_i)\cap V(S)$ is its terminal corresponding to $z_i$, for
$i\in\{1,2\}$. The terminals are disjoint cliques and unions of whole bags.
The markers are auxiliary vertices and do not belong to the final
composition.

A \emph{spot} has vertices $x,z_1,z_2$, where $x$ is strongly adjacent to
$z_1,z_2$, and $z_1,z_2$ are strongly antiadjacent. Its markers are
$z_1,z_2$. A spot is not a stripe because $x$ has two neighbors in the
marker set~\cite[Section~8]{ChudnovskySeymour}.

Chudnovsky and Seymour~\cite[Section~3]{ChudnovskySeymour} introduced eight
basic classes, denoted by $\cS_0,\ldots,\cS_7$. In the form of their structure
theorem used here, $\cS_0$, $\cS_2$, $\cS_4$, $\cS_5$, and $\cS_6$ are
handled by the three-clique or strip-structure outcomes. Only
thickenings of members of
$\cS_1\cup\cS_3\cup\cS_7$ occur as a separate basic
outcome~\cite[Theorem~7.2]{ChudnovskySeymour}. The classes $\cS_1$, $\cS_3$,
and $\cS_7$ are, respectively, the icosahedral, long circular interval, and
antiprismatic classes. Their definitions will be given in
Section~\ref{sec:basic-classes}.

Chudnovsky and Seymour~\cite[Section~7]{ChudnovskySeymour} also defined
fifteen stripe classes $\cZ_1,\ldots,\cZ_{15}$ and the subclass $\cZ_0$ of
members that are not thickenings of smaller members of these classes. The
marker specifications in their definitions show that members of
$\cZ_1,\ldots,\cZ_4$ have two markers, members of
$\cZ_6,\ldots,\cZ_{15}$ have one marker, and members of $\cZ_5$ have one or
two markers. Consequently, every member of $\cZ_0$ with two markers belongs
to one of $\cZ_1,\ldots,\cZ_5$.

We next recall strip-structures. Let $T$ be a trigraph, let
$Y\subseteq V(T)$, and let $X_1,\ldots,X_k\subseteq Y$. Following
Chudnovsky and Seymour~\cite[Section~7]{ChudnovskySeymour}, the family
$(X_i:1\leq i\leq k)$ is a \emph{circus in $Y$} if

\begin{enumerate}[label=\textup{(CS\arabic*)}]
\item for every $i$ and every $x\in X_i$, the neighbors of $x$ in
$Y\setminus X_i$ form a strong clique;

\item for $i<j$, the set $X_i\cap X_j$ is strongly anticomplete to
$Y\setminus(X_i\cup X_j)$;

\item no three distinct members of the family have a common vertex.
\end{enumerate}

A hypergraph $\mathcal H$ consists of a finite vertex set, a finite nonempty
edge set, and an incidence relation. The vertices incident with a hyperedge
$F$ are its \emph{ends}. For $r\in V(\mathcal H)$, let
$E_{\mathcal H}(r)$ denote the set of hyperedges having end $r$. A
\emph{strip-structure} $(\mathcal H,\eta)$ of $T$ assigns a set
$\eta(F)\subseteq V(T)$ to every $F\in E(\mathcal H)$ and a set
$\eta(F,r)\subseteq\eta(F)$ to every incidence of $F$ and an end $r$, such
that

\begin{enumerate}[label=\textup{(SD\arabic*)}]
\item the sets $\eta(F)$ are nonempty, pairwise disjoint, and have union
$V(T)$;

\item for every $r\in V(\mathcal H)$, the set
$H_r=\bigcup_{F\in E_{\mathcal H}(r)}\eta(F,r)$ is a strong clique;

\item if $F_1\neq F_2$, $x_i\in\eta(F_i)$ for $i\in\{1,2\}$, and $x_1$ and
$x_2$ are adjacent, then $F_1,F_2$ have a common end $r$ such that
$x_i\in\eta(F_i,r)$ for $i\in\{1,2\}$;

\item for every $F$, the family $(\eta(F,r):r\text{ is an end of }F)$ is a
circus in $\eta(F)$.
\end{enumerate}

If the ends of $F$ are $r_1,\ldots,r_k$, its \emph{associated strip} is
$(J,Z)$, where $J$ is obtained from $T[\eta(F)]$ by adding new vertices
$Z=\{z_1,\ldots,z_k\}$, and $z_i$ is strongly complete to
$\eta(F,r_i)$ and strongly anticomplete to
$V(J)\setminus(\{z_i\}\cup\eta(F,r_i))$. Thus the number of markers equals
the number of ends of $F$. The strip-structure is \emph{nontrivial} if it has
at least two hyperedges.

The \emph{nullity} of $(\mathcal H,\eta)$ is the number of incidences
$(F,r)$ for which $\eta(F,r)=\emptyset$. For a fixed trigraph $T$, a
strip-structure is \emph{optimal} if it has the maximum number of hyperedges
and, subject to this, minimum nullity. It is \emph{purified} if, for each
hyperedge $F$, either its attachment sets $\eta(F,r)$ are pairwise disjoint,
or $F$ has two ends $r_1,r_2$, $|\eta(F)|=1$, and
$\eta(F,r_1)=\eta(F,r_2)=\eta(F)$~\cite[Section~8]{ChudnovskySeymour}.

We finish the terminology by defining the graph composition used below. A
\emph{loopless multigraph} may have parallel edges but has no loops. Let $R$
be a loopless multigraph with at least two edges, and, for $r\in V(R)$, let
$E_R(r)$ denote the set of edges incident with $r$. A graph $G$ is a
\emph{composition over $R$} if its vertex set has a partition
$(V_e:e\in E(R))$ satisfying the following conditions.

\begin{enumerate}[label=\textup{$(\roman*)$}]
\item If $e=rs$ is a \emph{spot edge}, then $V_e=\{p_e\}$ and
$A_e^r=A_e^s=\{p_e\}$.

\item If $e=rs$ is not a spot edge, then $G[V_e]$ is a piece of type
$\cZ_1,\ldots,\cZ_5$ with nonempty disjoint terminals $A_e^r$ and $A_e^s$.

\item For every $r\in V(R)$, the set
$H_r=\bigcup_{e\in E_R(r)}A_e^r$ is a clique.

\item If $e\neq f$, $x\in V_e$, and $y\in V_f$, then $xy\in E(G)$ if and
only if $e$ and $f$ have a common end $r$ such that $x\in A_e^r$ and
$y\in A_f^r$.
\end{enumerate}

The clique $H_r$ is the \emph{hub} at $r$.

\subsection{The structure theorem of claw-free graphs of Chudnovsky and Seymour}

We use the following form of the global structure theorem of Chudnovsky and
Seymour~\cite[Theorems~7.2 and~8.1 and the proofs of Theorems~7.2
and~9.1]{ChudnovskySeymour}. In these proofs, the strip-structure used in
Theorem~7.2 is chosen optimal; Theorem~8.1 then gives nullity zero and
purification.

\begin{theorem}[Chudnovsky--Seymour~\cite{ChudnovskySeymour}]
\label{thm:cs-reduction}
Let $T$ be a connected claw-free trigraph whose vertex set is not the union
of three strong cliques. Then one of the following holds:

\begin{enumerate}[label=\textup{$(\roman*)$}]
\item $T$ is a thickening of a member of
$\cS_1\cup\cS_3\cup\cS_7$;

\item $T$ has an optimal nontrivial strip-structure with nullity zero that
is purified, and every associated strip $(J,Z)$ satisfies
$1\leq |Z|\leq2$ and is either a three-vertex strip with $|Z|=2$, or a
thickening of a member of $\cZ_0$.
\end{enumerate}
\end{theorem}

\subsection{Reduction of claw-free graphs with no simplicial vertex and no clique cutset}

We now derive the form used in the proof of the main theorem.

\begin{theorem}\label{thm:structure}
Let $G$ be a connected claw-free graph with no simplicial vertex and no
clique cutset. Then at least one of the following holds:

\begin{enumerate}[label=\textup{$(\roman*)$}]
\item $V(G)$ is the union of three cliques;

\item $G$ is a thickening of a member of
$\cS_1\cup\cS_3\cup\cS_7$;

\item $G$ is a composition over a loopless multigraph with at least two
edges, and every non-spot piece has type $\cZ_1,\ldots,\cZ_5$.
\end{enumerate}
\end{theorem}

\begin{proof}
Suppose that $V(G)$ is not the union of three cliques. Apply
Theorem~\ref{thm:cs-reduction}. Its first outcome gives $(ii)$. We may
therefore assume that $G$ has an optimal nontrivial strip-structure
$(\mathcal H,\eta)$ as in its second outcome.

We first show that every hyperedge of $\mathcal H$ has two ends. Suppose
that a hyperedge $F$ has only one end $r$, and put $V_F=\eta(F)$ and
$A_F=\eta(F,r)$. Since the nullity is zero, $A_F\neq\emptyset$. Moreover,
$A_F$ is contained in the hub $H_r$ and hence is a clique. All edges between
$V_F$ and $V(G)\setminus V_F$ have their end in $V_F$ inside $A_F$. Since
the strip-structure is nontrivial, $V(G)\setminus V_F\neq\emptyset$.

If $V_F\setminus A_F\neq\emptyset$, then $A_F$ is a clique cutset separating
$V_F\setminus A_F$ from $V(G)\setminus V_F$, a contradiction. Hence
$V_F=A_F$. For every $x\in V_F$, all neighbors of $x$ lie in $H_r$. Since
$x\in H_r$ and $H_r$ is a clique, $N_G[x]$ is a clique. Thus $x$ is
simplicial, again a contradiction. Therefore every hyperedge has two ends.

It follows that $\mathcal H$ is a loopless multigraph with at least two
edges. Every associated strip has two markers. A three-vertex strip is a
spot, since its unique non-marker vertex belongs to both nonempty attachment
sets. Every other strip is a thickening of a two-marker member of $\cZ_0$,
and hence of a member of $\cZ_1,\ldots,\cZ_5$. The strip-structure axioms now
give exactly the partition, hubs, and cross-edge rule in the definition of a
composition. Thus $(iii)$ holds. This completes the proof of
Theorem~\ref{thm:structure}.
\end{proof}

\section{The union of cliques and thickenings}
\label{sec:basic-classes}
In this section, we treat outcomes $(i)$ and $(ii)$ of
Theorem~\ref{thm:structure}. Thus we will handle graphs whose vertex set is
the union of three cliques and then consider thickenings of members of
$\cS_1$, $\cS_3$, and $\cS_7$.
We find a perfect$\Omega_G$-transversal when the vertex set  is covered by three
cliques.

\begin{lemma}\label{lem:three-cliques}
Let $G$ be a graph with at least one edge. If $V(G)$ is the union of three cliques, then $G$ has a perfect $\Omega_G$-transversal.
\end{lemma}

\begin{proof}
Let $V(G)=A\cup B\cup C$, where $A,B,C$ are cliques. We may assume that $A\neq\emptyset$, and choose $x\in A$. Since $x$ is adjacent to every vertex of $A\setminus\{x\}$, we have $M_G(x)\subseteq B\cup C$. Thus $M_G(x)$ is the union of two cliques, so $G[M_G(x)]$ is perfect.

Put $S=M_G(x)\cup\{x\}$. Since $x$ is anticomplete to $M_G(x)$, the graph $G[S]$ is perfect. If some $K\in\Omega_G$ misses $S$, then $K\subseteq N_G(x)$, and hence $K\cup\{x\}$ is a clique larger than $K$, a contradiction. Therefore $S$ is a perfect $\Omega_G$-transversal.
\end{proof}

We next need some tools for the three basic thickening classes in
outcome~\textup{(ii)} of Theorem~\ref{thm:structure}. We first recall
linear interval trigraphs. Following Chudnovsky and
Seymour~\cite[Section~7]{ChudnovskySeymour}, a trigraph $T$ is a
\emph{linear interval trigraph} if its vertices admit an ordering
$v_1,\ldots,v_n$ such that, whenever $1\leq i<j<k\leq n$ and $v_i$
is adjacent to $v_k$, the vertex $v_j$ is strongly adjacent to both
$v_i$ and $v_k$. Restricting such an ordering to any subset of
$V(T)$ shows that every induced subtrigraph of a linear interval
trigraph is again a linear interval trigraph.
Figure~\ref{fig:linear-interval} shows a linear interval trigraph.

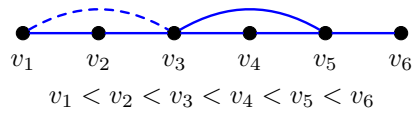
\begin{figure}[H]
    \centering
    \tikzstyle{v}=[circle, draw, fill=black, inner sep=0pt,
    minimum size=5pt]

    \begin{tikzpicture}[
        scale=1.0,
        baseline=(current bounding box.center)
    ]
        \useasboundingbox (-3.30,-1.05) rectangle (3.30,1.45);

        \node[v] (v1) at (-2.50,0) {};
        \node[v] (v2) at (-1.50,0) {};
        \node[v] (v3) at (-0.50,0) {};
        \node[v] (v4) at (0.50,0) {};
        \node[v] (v5) at (1.50,0) {};
        \node[v] (v6) at (2.50,0) {};

        \node at (-2.50,-0.38) {$v_1$};
        \node at (-1.50,-0.38) {$v_2$};
        \node at (-0.50,-0.38) {$v_3$};
        \node at (0.50,-0.38) {$v_4$};
        \node at (1.50,-0.38) {$v_5$};
        \node at (2.50,-0.38) {$v_6$};

        % Strong adjacencies between consecutive vertices
        \draw[blue, line width=0.9pt]
            (v1)--(v2)
            (v2)--(v3)
            (v3)--(v4)
            (v4)--(v5)
            (v5)--(v6);

        % Semiadjacency
        \draw[blue, dashed, line width=0.9pt]
            (v1) to[bend left=30] (v3);

        % Strong adjacency
        \draw[blue, line width=0.9pt]
            (v3) to[bend left=30] (v5);

        \node at (0,-0.85)
            {$v_1<v_2<v_3<v_4<v_5<v_6$};
    \end{tikzpicture}

    \caption{A linear interval trigraph. Solid edges represent strong
    adjacencies, the dashed edge $v_1v_3$ represents a semiadjacency,
    and all unrepresented pairs are strongly antiadjacent.}
    \label{fig:linear-interval}
\end{figure}

To connect these trigraphs with perfect graphs, we use the standard trigraph
conventions of Chudnovsky and Seymour~\cite[Section~2]
{ChudnovskySeymour}. Let $k\geq4$. A \emph{hole} in a trigraph $T$ is an ordering
$v_1,\ldots,v_k$ of distinct vertices such that consecutive vertices,
with indices taken modulo $k$, are adjacent and all other pairs are
antiadjacent. An \emph{antihole} is defined in the same way with
adjacency and antiadjacency interchanged. A trigraph is \emph{Berge} if
it contains no odd hole and no odd antihole of length at least five.
By Theorem~\ref{thm:SPGT}, an ordinary graph is perfect if and only if
it is Berge. Chudnovsky and Plumettaz~\cite[Lemmas~5.3 and~6.4]
{ChudnovskyPlumettaz} proved that every linear interval trigraph is
Berge and that the Berge property is preserved under thickening.

\begin{lemma}[Chudnovsky-- Plumettaz~\cite{ChudnovskyPlumettaz}]\label{lem:linear-thickening}
Every graph that is a thickening of a linear interval trigraph is
perfect.
\end{lemma}

We next consider the circular analogue of linear interval trigraphs.
Let $\Sigma$ be a circle, and let $F_1,\ldots,F_k$ be closed arcs of
$\Sigma$ such that no two arcs have a common endpoint and no three arcs
have union $\Sigma$. Let $V$ be a finite subset of $\Sigma$. Two
distinct points of $V$ are strongly antiadjacent if they lie in no
common arc. They are strongly adjacent if they lie in a common arc and
at least one of them belongs to the interior of a common arc. Otherwise,
they are semiadjacent. A trigraph obtained in this way is called a
\emph{long circular interval trigraph}. Let $\mathcal S_3$ denote the
class of all long circular interval
trigraphs~\cite[Section~3]{ChudnovskySeymour}. Figure~\ref{fig:long-circular-interval} gives a circular representation
of a member of $\mathcal S_3$.

\begin{figure}[H]
    \centering
    \tikzstyle{v}=[circle, draw, fill=black, inner sep=0pt,
    minimum size=5pt]

    \begin{tikzpicture}[
        scale=1.0,
        baseline=(current bounding box.center)
    ]
        \useasboundingbox (-2.75,-2.65) rectangle (2.75,2.75);

        % The circle Sigma
        \draw[black, line width=0.7pt] (0,0) circle (1.65);

        % Vertices on Sigma
        \node[v] (v1) at (20:1.65) {};
        \node[v] (v2) at (55:1.65) {};
        \node[v] (v3) at (110:1.65) {};
        \node[v] (v4) at (145:1.65) {};
        \node[v] (v5) at (200:1.65) {};
        \node[v] (v6) at (235:1.65) {};
        \node[v] (v7) at (290:1.65) {};
        \node[v] (v8) at (325:1.65) {};

        % Vertex labels
        \node at (20:1.30) {$v_1$};
        \node at (55:1.30) {$v_2$};
        \node at (110:1.30) {$v_3$};
        \node at (145:1.30) {$v_4$};
        \node at (200:1.30) {$v_5$};
        \node at (235:1.30) {$v_6$};
        \node at (290:1.30) {$v_7$};
        \node at (325:1.30) {$v_8$};

        % Arcs, drawn at different distances for clarity
        \draw[blue, line width=1.1pt]
            (20:1.82)
            arc[start angle=20, end angle=145, radius=1.82];

        \draw[blue, line width=1.1pt]
            (110:1.96)
            arc[start angle=110, end angle=235, radius=1.96];

        \draw[blue, line width=1.1pt]
            (200:2.10)
            arc[start angle=200, end angle=325, radius=2.10];

        \draw[blue, line width=1.1pt]
            (290:2.24)
            arc[start angle=290, end angle=415, radius=2.24];

        % Connect the arc endpoints to the corresponding vertices
        \draw[gray, dotted]
            (20:1.65)--(20:1.82)
            (145:1.65)--(145:1.82)
            (110:1.65)--(110:1.96)
            (235:1.65)--(235:1.96)
            (200:1.65)--(200:2.10)
            (325:1.65)--(325:2.10)
            (290:1.65)--(290:2.24)
            (55:1.65)--(55:2.24);

        % Arc labels
        \node at (82.5:2.04) {$F_1$};
        \node at (172.5:2.18) {$F_2$};
        \node at (262.5:2.32) {$F_3$};
        \node at (352.5:2.46) {$F_4$};

        \node at (0,0) {$\Sigma$};
    \end{tikzpicture}

    \caption{A circular representation of a long circular interval
    trigraph. The four arcs are drawn at different distances from
    $\Sigma$ for clarity. Their endpoint pairs are semiadjacent, while
    two vertices lying in no common arc are strongly antiadjacent.}
    \label{fig:long-circular-interval}
\end{figure}
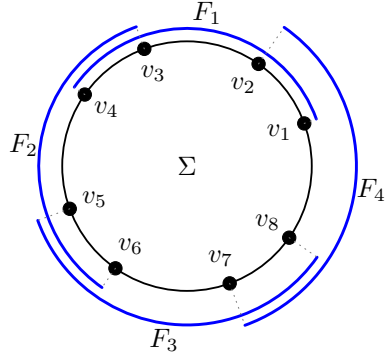
The following lemma gives the property of long circular interval
trigraphs that we need.

\begin{lemma}\label{lem:circular-antineighborhood}
Let $G$ be a thickening of a long circular interval trigraph. Then
$G[M_G(x)]$ is perfect for every $x\in V(G)$. In particular, if $G$
has at least one edge, then $G$ has a perfect $\Omega_G$-transversal.
\end{lemma}

\begin{proof}
Let $T$ be the base trigraph of $G$, with bags
$(X_v:v\in V(T))$, and let $x\in X_v$. Let $R$ be the set of vertices
of $T$ that are strongly antiadjacent to $v$. If $v$ has a
semiadjacent neighbor $u$, put
$Y=\{y\in X_u:xy\notin E(G)\}$; otherwise, put $Y=\emptyset$. Then
$M_G(x)=Y\cup\bigcup_{s\in R}X_s$.

Let $R'=R$ if $Y=\emptyset$, and let $R'=R\cup\{u\}$ otherwise. Cut
the circle at $v$ and order the vertices of $R'$ along the resulting
line. We show that $T[R']$ is a linear interval trigraph.

Let $a<b<c$ in this order, and suppose that $a$ and $c$ lie in a
common arc $F$. The arc $F$ does not contain $v$. This is clear if
$a,c\in R$. Suppose that one of $a,c$ is $u$. Since $u$ and $v$ are
semiadjacent, every arc containing both of them has $u$ and $v$ as
its endpoints. Every other point of such an arc is strongly adjacent
to $v$, whereas the other one of $a,c$ belongs to $R$. Thus $F$
cannot contain $v$.

Since $F$ avoids the cut point and contains $a$ and $c$, it contains
$b$ in its interior. Hence $b$ is strongly adjacent to both $a$ and
$c$. Therefore $T[R']$ is a linear interval trigraph.

The semiadjacent pairs of $T$ form a matching, so $v$ is the only
possible semiadjacent neighbor of $u$. Therefore, when $Y\neq
\emptyset$, the bag $X_u$ may be replaced by $Y$ without changing the
relations with the other retained bags. It follows that $G[M_G(x)]$
is a thickening of $T[R']$. Hence $G[M_G(x)]$ is perfect by
Lemma~\ref{lem:linear-thickening}.

Finally, put $S=M_G(x)\cup\{x\}$. The graph $G[S]$ is perfect because
$x$ is anticomplete to $M_G(x)$. If a maximum clique $K$ of $G$
misses $S$, then $K\subseteq N_G(x)$, so $K\cup\{x\}$ is a larger
clique. Therefore $S$ is a perfect $\Omega_G$-transversal.
\end{proof}

We next introduce antiprismatic trigraphs. A trigraph $T$ is
\emph{antiprismatic} if, for every set $X\subseteq V(T)$ with
$|X|=4$, the subtrigraph $T[X]$ is not a claw and $X$ contains at
least two strongly adjacent pairs. Let $\mathcal S_7$ denote the class
of all antiprismatic trigraphs~\cite[Section~3]{ChudnovskySeymour}.

The anti-neighborhoods of thickenings of antiprismatic trigraphs are
also perfect. In this case, their complements are bipartite.

\begin{lemma}\label{lem:antiprismatic-antineighborhood}
Let $G$ be a thickening of an antiprismatic trigraph. Then
$G[M_G(x)]$ is perfect for every $x\in V(G)$. In particular, if $G$
has at least one edge, then $G$ has a perfect $\Omega_G$-transversal.
\end{lemma}

\begin{proof}
Let $T$ be the base trigraph of $G$, with bags
$(X_v:v\in V(T))$, and let $x\in X_v$. Let $R$ be the set of vertices
of $T$ that are strongly antiadjacent to $v$. If $v$ has a
semiadjacent neighbor $u$, put
$Y=\{y\in X_u:xy\notin E(G)\}$; otherwise, put $Y=\emptyset$. Then
$M_G(x)=Y\cup\bigcup_{s\in R}X_s$.

Let $R'=R$ if $Y=\emptyset$, and let $R'=R\cup\{u\}$ otherwise.
Define a graph $Q$ on $R'$ by joining two vertices when they are not
strongly adjacent in $T$. We show that $Q$ is bipartite.

First, $Q[R]$ is a matching. Otherwise, some $s\in R$ has two
distinct neighbors $p,q\in R$. Among the four vertices
$\{v,s,p,q\}$, no pair containing $v$ is strongly adjacent, and
neither $sp$ nor $sq$ is strongly adjacent. Thus these four vertices
contain at most one strongly adjacent pair, contrary to the definition
of an antiprismatic trigraph.

Suppose that $u\in R'$. Since the semiadjacent pairs of $T$ form a
matching, $v$ is the only semiadjacent neighbor of $u$. If
$u,s,t$ form a triangle in $Q$, then the four vertices
$\{v,u,s,t\}$ contain no strongly adjacent pair, again a
contradiction. Hence $Q$ has no triangle containing $u$.

A cycle of $Q$ avoiding $u$ cannot exist because $Q[R]$ is a
matching. A cycle containing $u$ would contain a path in $Q[R]$
between two neighbors of $u$. Such a path consists of a single edge,
so the cycle would be a triangle. Therefore $Q$ is bipartite.

Every bag retained in $M_G(x)$ is an independent set in
$\overline{G[M_G(x)]}$. Moreover, an edge of the complement between
two retained bags can occur only when their base vertices are adjacent
in $Q$. Since $Q$ is bipartite, $\overline{G[M_G(x)]}$ is bipartite.
Hence $G[M_G(x)]$ is perfect.

Finally, put $S=M_G(x)\cup\{x\}$. The graph $G[S]$ is perfect because
$x$ is anticomplete to $M_G(x)$. If a maximum clique $K$ of $G$
misses $S$, then $K\subseteq N_G(x)$, so $K\cup\{x\}$ is a larger
clique. Therefore $S$ is a perfect $\Omega_G$-transversal.
\end{proof}
It remains to consider the icosahedral trigraphs. The
\emph{icosahedron} is the unique planar graph on twelve vertices in
which every vertex has degree five. Let $T_0$ have vertices
$v_0,v_1,\ldots,v_{11}$. For $1\leq i\leq10$, the vertex $v_i$ is
adjacent to $v_{i+1}$ and $v_{i+2}$, where the subscripts are taken
modulo $10$. Moreover, $v_0$ is adjacent to
$\{v_1,v_3,v_5,v_7,v_9\}$, and $v_{11}$ is adjacent to
$\{v_2,v_4,v_6,v_8,v_{10}\}$. All these adjacencies are strong.

Let $T_1=T_0-v_{11}$. Further members are obtained from
$T_1-v_{10}$ by possibly changing either or both of the pairs
$v_1v_4$ and $v_6v_9$ from strong antiadjacencies to semiadjacencies.
The class consisting of $T_0$, $T_1$, and all the resulting trigraphs
is denoted by $\mathcal S_1$~\cite[Section~3]{ChudnovskySeymour}. Figure~\ref{fig:icosahedral} shows the member $T_1$ of
$\mathcal S_1$.

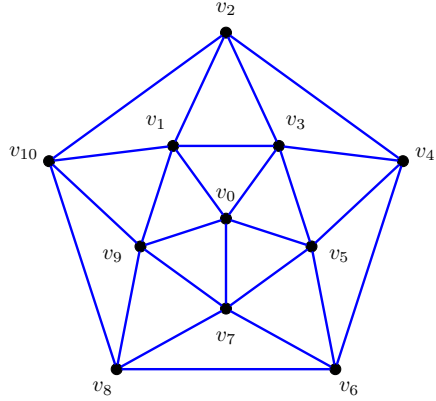
\begin{figure}[H]
    \centering
    \tikzstyle{v}=[circle, draw, fill=black, inner sep=0pt,
    minimum size=5pt]

   \begin{tikzpicture}[
    scale=0.82,
    transform shape,
    baseline=(current bounding box.center)
]
        \useasboundingbox (-3.65,-3.50) rectangle (3.65,3.65);

        % Outer pentagon: even-indexed vertices
        \node[v] (v2)  at (90:3.00) {};
        \node[v] (v4)  at (18:3.00) {};
        \node[v] (v6)  at (-54:3.00) {};
        \node[v] (v8)  at (-126:3.00) {};
        \node[v] (v10) at (162:3.00) {};

        % Inner pentagon: odd-indexed vertices
        \node[v] (v1) at (126:1.45) {};
        \node[v] (v3) at (54:1.45) {};
        \node[v] (v5) at (-18:1.45) {};
        \node[v] (v7) at (-90:1.45) {};
        \node[v] (v9) at (-162:1.45) {};

        % Central vertex
        \node[v] (v0) at (0,0) {};

        % Outer pentagon
        \draw[blue, line width=0.9pt]
            (v2)--(v4)--(v6)--(v8)--(v10)--(v2);

        % Inner pentagon
        \draw[blue, line width=0.9pt]
            (v1)--(v3)--(v5)--(v7)--(v9)--(v1);

        % Edges between the two pentagons
        \draw[blue, line width=0.9pt]
            (v1)--(v10)
            (v1)--(v2)
            (v3)--(v2)
            (v3)--(v4)
            (v5)--(v4)
            (v5)--(v6)
            (v7)--(v6)
            (v7)--(v8)
            (v9)--(v8)
            (v9)--(v10);

        % Edges incident with v_0
        \draw[blue, line width=0.9pt]
            (v0)--(v1)
            (v0)--(v3)
            (v0)--(v5)
            (v0)--(v7)
            (v0)--(v9);

        % Labels of the outer vertices
        \node at (90:3.38) {$v_2$};
        \node at (18:3.38) {$v_4$};
        \node at (-54:3.38) {$v_6$};
        \node at (-126:3.38) {$v_8$};
        \node at (162:3.43) {$v_{10}$};

        % Labels of the inner vertices
        \node at (126:1.92) {$v_1$};
        \node at (54:1.92) {$v_3$};
        \node at (-18:1.92) {$v_5$};
        \node at (-90:1.92) {$v_7$};
        \node at (-162:1.92) {$v_9$};

        % Label of the central vertex
        \node at (90:0.38) {$v_0$};
    \end{tikzpicture}

    \caption{The icosahedral graph $T_1=T_0-v_{11}$, which belongs
    to $\mathcal S_1$. Every displayed edge is a strong adjacency,
    and every unrepresented pair is strongly antiadjacent.}
    \label{fig:icosahedral}
\end{figure}
We use the following consequence of the substitution theorem of
Lov\'asz~\cite{LovaszSubstitution}: Every clique blow-up of a perfect
graph is perfect. This allows us to find a perfect transversal in every
thickening of an icosahedral trigraph.

\begin{lemma}\label{lem:icosahedral}
Every graph $G$ that is a thickening of a member of $\cS_1$ has a
perfect $\Omega_G$-transversal.
\end{lemma}

\begin{proof}
Let $T$ be the base trigraph of $G$, with bags
$(X_v:v\in V(T))$. Suppose first that $T\in\{T_0,T_1\}$. Put
$C_0=\{v_0,v_2,v_6\}$, $C_1=\{v_1,v_4,v_8\}$,
$C_2=\{v_3,v_7,v_{10}\}$, and
$C_3=\{v_5,v_9,v_{11}\}$, where $v_{11}$ is omitted when $T=T_1$.
Each $C_i$ is an independent set. Let
$P=\bigcup_{v\in C_0\cup C_1}X_v$. Since
$T[C_0\cup C_1]$ is bipartite, $G[P]$ is perfect by the substitution
theorem of Lov\'asz.

We show that every maximal clique of $T_0$ and $T_1$ is a triangle.
The vertices $v_1,\ldots,v_{10}$ induce the square of a $10$-cycle,
whose clique number is three. The neighbors of $v_0$ among these
vertices induce a $5$-cycle, and the same holds for the neighbors of
$v_{11}$. Moreover, $v_0$ and $v_{11}$ are nonadjacent. Hence $T_0$
is $K_4$-free.

Every edge of $T_0$ lies in two triangles. Indeed, with subscripts
taken modulo $10$, an edge $v_iv_{i+1}$ extends to triangles through
$v_{i-1}$ and $v_{i+2}$, while an edge $v_iv_{i+2}$ extends through
$v_{i+1}$ and through $v_0$ or $v_{11}$ according as $i$ is odd or
even. The same conclusion is immediate for edges incident with $v_0$
or $v_{11}$. Thus every maximal clique of $T_0$ is a triangle. After
$v_{11}$ is deleted, every remaining edge still lies in a triangle,
so the same holds for $T_1$.

Let $K$ be a maximum clique of $G$. The bags met by $K$ are indexed
by a maximal clique of $T$, and hence by a triangle. Since each $C_i$
is independent, this triangle meets three distinct sets among
$C_0,C_1,C_2,C_3$. It must therefore meet $C_0\cup C_1$, and hence
$K\cap P\neq\emptyset$.

It remains to consider a trigraph obtained from $T_1-v_{10}$ by
possibly changing one or both of $v_1v_4$ and $v_6v_9$ to
semiadjacent pairs. Put $D_0=\{v_0,v_4,v_8\}$,
$D_1=\{v_1,v_6\}$, $D_2=\{v_3,v_7\}$, and
$D_3=\{v_2,v_5,v_9\}$. Let
$P=\bigcup_{v\in D_0\cup D_3}X_v$. The sets $D_0$ and $D_3$ are
independent, and at most one end of each possible semiadjacent pair
belongs to $D_0\cup D_3$. Hence $G[P]$ is a clique blow-up of a
bipartite graph and is perfect.

Any clique of $G$ that misses $P$ is contained in
$X_{v_1}\cup X_{v_3}$ or in $X_{v_6}\cup X_{v_7}$. The first set is
complete to the nonempty bag $X_{v_0}$, and the second is complete to
the nonempty bag $X_{v_5}$. Such a clique is not maximal. Therefore
every maximum clique of $G$ meets $P$, and $P$ is a perfect
$\Omega_G$-transversal.
\end{proof}

\section{Perfect transversals for the two-marker stripe classes}\label{sec:stripes}

It remains to prepare for outcome $(iii)$ of
Theorem~\ref{thm:structure}. The marker-count argument in
Section~\ref{sec:main-claw-free} shows that every non-spot piece arising in
this outcome is a thickening of a two-marker member of one of
$\cZ_1,\ldots,\cZ_5$. We now define these five types and prove the local
properties needed later. For $\cZ_1,\ldots,\cZ_4$, deleting either
terminal leaves a perfect graph. The class $\cZ_5$ requires a separate
central perfect set. Lemma~\ref{lem:oriented-local-set} will combine
these properties into one orientation-compatible statement.

\subsection{Type \texorpdfstring{$\cZ_1$}{Z1}}

Following Chudnovsky and Seymour~\cite[Section~7]{ChudnovskySeymour},
the first type is built from a linear interval trigraph. A stripe
$(T,\{v_1,v_n\})$ belongs to $\cZ_1$ if $T$ is a linear interval
trigraph with ordering $v_1,\ldots,v_n$, where $n\geq2$, and the
following conditions hold:

\begin{enumerate}[label=\textup{$(\roman*)$}]
\item $v_1$ and $v_n$ are strongly antiadjacent;
\item no vertex is adjacent to both $v_1$ and $v_n$;
\item no vertex is semiadjacent to $v_1$ or $v_n$.
\end{enumerate}

The markers are $v_1$ and $v_n$. A representative member of
$\cZ_1$ is shown in Figure~\ref{fig:z1}. In this example, the base is
the square of a seven-vertex path. The terminals corresponding to
$v_1$ and $v_7$ are represented by $\{v_2,v_3\}$ and
$\{v_5,v_6\}$, respectively.

\begin{figure}[H]
    \centering
    \tikzstyle{v}=[
        circle,
        draw,
        fill=black,
        inner sep=0pt,
        minimum size=5pt
    ]

    \begin{tikzpicture}[
        scale=0.82,
        transform shape,
        baseline=(current bounding box.center)
    ]
        \useasboundingbox (-0.8,-1.40) rectangle (7.75,1.35);

        \node[v] (v1) at (0,0) {};
        \node[v] (v2) at (1.15,0) {};
        \node[v] (v3) at (2.30,0) {};
        \node[v] (v4) at (3.45,0) {};
        \node[v] (v5) at (4.60,0) {};
        \node[v] (v6) at (5.75,0) {};
        \node[v] (v7) at (6.90,0) {};

        % Consecutive pairs
        \draw[blue, line width=0.9pt]
            (v1)--(v2)--(v3)--(v4)--(v5)--(v6)--(v7);

        % Pairs at distance two
        \draw[blue, line width=0.9pt]
            (v1) to[bend left=35] (v3)
            (v2) to[bend left=35] (v4)
            (v3) to[bend left=35] (v5)
            (v4) to[bend left=35] (v6)
            (v5) to[bend left=35] (v7);

        \node at (0,-0.42) {$z_1=v_1$};
        \node at (1.15,-0.42) {$v_2$};
        \node at (2.30,-0.42) {$v_3$};
        \node at (3.45,-0.42) {$v_4$};
        \node at (4.60,-0.42) {$v_5$};
        \node at (5.75,-0.42) {$v_6$};
        \node at (6.90,-0.42) {$z_2=v_7$};

        \draw[
            decorate,
            decoration={brace,mirror,amplitude=4pt},
            line width=0.7pt
        ]
            (0.92,-0.68)--(2.53,-0.68)
            node[midway,below=5pt] {$A_1$};

        \draw[
            decorate,
            decoration={brace,mirror,amplitude=4pt},
            line width=0.7pt
        ]
            (4.37,-0.68)--(5.98,-0.68)
            node[midway,below=5pt] {$A_2$};
    \end{tikzpicture}

    \caption{A representative stripe of type $\cZ_1$.}
    \label{fig:z1}
\end{figure}
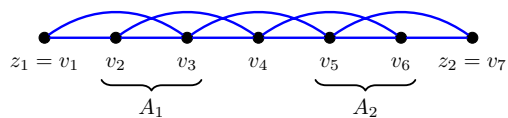

The linear ordering in the definition gives the perfectness property
needed later.

\begin{lemma}\label{lem:z1}
Let $S$ be the piece of a thickening of a stripe of type $\cZ_1$, and
let $A_1,A_2$ be its terminals. Then both $S-A_1$ and $S-A_2$ are
perfect.
\end{lemma}

\begin{proof}
Let $T$ be the base trigraph. After the marker bags are deleted, $S$
is a thickening of an induced subtrigraph of $T$. Each terminal is a
union of whole bags. Hence deleting either terminal produces a
thickening of another induced subtrigraph of $T$. Since every induced
subtrigraph of a linear interval trigraph is again a linear interval
trigraph, the result follows from Lemma~\ref{lem:linear-thickening}.
\end{proof}

\subsection{Type \texorpdfstring{$\cZ_2$}{Z2}}

Following Chudnovsky and Seymour~\cite[Section~7]{ChudnovskySeymour},
the second type has three main strong cliques. Fix $n\geq2$, and let
$A^0=\{a_0,a_1,\ldots,a_n\}$,
$B^0=\{b_0,b_1,\ldots,b_n\}$, and
$C^0=\{c_1,\ldots,c_n\}$ be pairwise disjoint strong cliques. The
remaining relations are defined as follows.

For $0\leq i,j\leq n$ with $(i,j)\neq(0,0)$, the vertices $a_i$ and
$b_j$ are adjacent if and only if $i=j$. The pair $a_0b_0$ is strongly
antiadjacent. For $1\leq i\leq n$ and $0\leq j\leq n$, the vertex
$c_i$ is adjacent to $a_j$ and to $b_j$ if and only if $i\neq j$ and
$j\neq0$. Every other pair not specified as adjacent is strongly
antiadjacent.

A stripe of type $\cZ_2$ is obtained by deleting a set
$D\subseteq(A^0\cup B^0\cup C^0)\setminus\{a_0,b_0\}$ such that
$|C^0\setminus D|\geq2$. The relations among the remaining vertices
are inherited from the preceding rules, except for the following
possible changes:

\begin{enumerate}[label=\textup{$(\alph*)$}]
\item if $b_i\in D$, the pair $a_ic_i$ may be semiadjacent instead
of strongly antiadjacent;
\item if $a_i\in D$, the pair $b_ic_i$ may be semiadjacent instead
of strongly antiadjacent;
\item if $c_i\in D$, the pair $a_ib_i$ may be semiadjacent instead
of strongly adjacent.
\end{enumerate}

In each of the three families above, at most one surviving pair may be
semiadjacent. Every other adjacent pair is strongly adjacent. The
markers are $a_0$ and $b_0$. Figure~\ref{fig:z2} shows the case
$n=2$ and $D=\emptyset$.

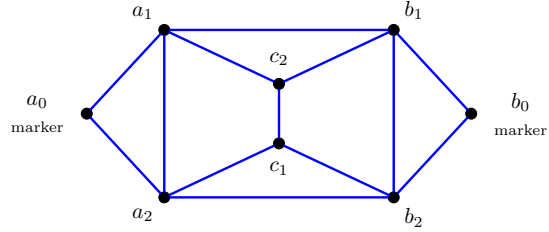
\begin{figure}[H]
    \centering
    \tikzstyle{v}=[
        circle,
        draw,
        fill=black,
        inner sep=0pt,
        minimum size=5pt
    ]

    \begin{tikzpicture}[
        scale=0.82,
        transform shape,
        baseline=(current bounding box.center)
    ]
        \useasboundingbox (-4.50,-2.15) rectangle (4.50,2.15);

        \node[v] (a0) at (-3.10,0) {};
        \node[v] (a1) at (-1.85,1.35) {};
        \node[v] (a2) at (-1.85,-1.35) {};

        \node[v] (b0) at (3.10,0) {};
        \node[v] (b1) at (1.85,1.35) {};
        \node[v] (b2) at (1.85,-1.35) {};

        \node[v] (c2) at (0,0.48) {};
        \node[v] (c1) at (0,-0.48) {};

        % The three strong cliques
        \draw[blue, line width=0.9pt]
            (a0)--(a1)--(a2)--(a0)
            (b0)--(b1)--(b2)--(b0)
            (c1)--(c2);

        % Remaining strong adjacencies
        \draw[blue, line width=0.9pt]
            (a1)--(b1)
            (a2)--(b2)
            (c2)--(a1)
            (c2)--(b1)
            (c1)--(a2)
            (c1)--(b2);

        % Vertex labels
        \node[left=7pt, align=center] at (a0)
            {$a_0$\\[-1pt]{\scriptsize marker}};
        \node[above left=2pt] at (a1) {$a_1$};
        \node[below left=2pt] at (a2) {$a_2$};

        \node[right=7pt, align=center] at (b0)
            {$b_0$\\[-1pt]{\scriptsize marker}};
        \node[above right=2pt] at (b1) {$b_1$};
        \node[below right=2pt] at (b2) {$b_2$};

\node[above=5pt] at (c2) {$c_2$};
\node[below=5pt] at (c1) {$c_1$};
    \end{tikzpicture}

    \caption{A representative stripe of type $\cZ_2$ with $n=2$.}
    \label{fig:z2}
\end{figure}

Although the adjacencies in $\cZ_2$ are more involved, deleting either
terminal leaves a graph covered by two cliques.

\begin{lemma}\label{lem:z2}
Let $S$ be the piece of a thickening of a stripe of type $\cZ_2$, and
let $A_1,A_2$ be its terminals. Then both $S-A_1$ and $S-A_2$ are
perfect.
\end{lemma}

\begin{proof}
Let $(X_v:v\in V(T))$ be the bags of the thickening, where $T$ is the
base trigraph. The terminal corresponding to $a_0$ is the union of
the bags indexed by the surviving vertices of
$A^0\setminus\{a_0\}$. Similarly, the terminal corresponding to $b_0$
is the union of the bags indexed by the surviving vertices of
$B^0\setminus\{b_0\}$.

After the first terminal is deleted, the remaining vertices are covered
by the bags indexed by $B^0\setminus\{b_0\}$ and by those indexed by
$C^0$. Each union is a clique because $B^0$ and $C^0$ are strong
cliques. Thus $S-A_1$ is cobipartite and therefore perfect. By
symmetry, $S-A_2$ is also cobipartite and perfect.
\end{proof}

\subsection{Type \texorpdfstring{$\cZ_3$}{Z3}}

Following Chudnovsky and Seymour~\cite[Section~7]{ChudnovskySeymour},
the third type is defined through a line trigraph. Let $R$ be a graph
containing a path $h_1h_2h_3h_4h_5$ in this order. Suppose that
$h_1$ and $h_5$ have degree one and that every edge of $R$ is incident
with at least one of $h_2,h_3,h_4$.

In the terminology of Chudnovsky and
Seymour~\cite[Sections~2 and~7]{ChudnovskySeymour}, a \emph{line
trigraph} of $R$ has vertex set $E(R)$. Two edges of $R$
with no common end are strongly antiadjacent. Two edges with a common
end are adjacent, and they are strongly adjacent if their common end
has degree at least three. If their only common end has degree two,
they may be strongly adjacent or semiadjacent, subject to the condition
that the semiadjacent pairs form a matching.

Put $p=h_2h_3$ and $q=h_3h_4$. Starting with a line trigraph of $R$,
make $p$ and $q$ either semiadjacent or strongly antiadjacent. The
markers are $h_1h_2$ and $h_4h_5$. Only choices for which these
vertices are the markers of a stripe are included in $\cZ_3$. In
particular, no vertex of the resulting trigraph is adjacent to both
markers.

Figure~\ref{fig:z3} gives an example. The underlying graph has two
additional edges $h_2x$ and $h_4y$. In the corresponding line
trigraph, the dashed pair is semiadjacent.

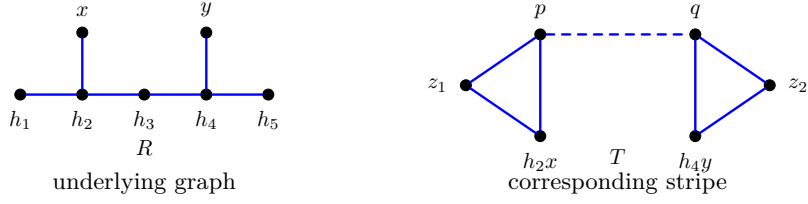
\begin{figure}[H]
    \centering
    \tikzstyle{v}=[
        circle,
        draw,
        fill=black,
        inner sep=0pt,
        minimum size=5pt
    ]

    \setlength{\tabcolsep}{2.2em}

    \begin{tabular}{cc}
        \begin{tikzpicture}[
            scale=0.82,
            transform shape,
            baseline=(current bounding box.center)
        ]
            \useasboundingbox (-2.65,-1.05) rectangle (2.65,1.55);

            \node[v] (h1) at (-2,0) {};
            \node[v] (h2) at (-1,0) {};
            \node[v] (h3) at (0,0) {};
            \node[v] (h4) at (1,0) {};
            \node[v] (h5) at (2,0) {};
            \node[v] (x) at (-1,1) {};
            \node[v] (y) at (1,1) {};

            \draw[blue, line width=0.9pt]
                (h1)--(h2)--(h3)--(h4)--(h5)
                (h2)--(x)
                (h4)--(y);

            \node[below=4pt] at (h1) {$h_1$};
            \node[below=4pt] at (h2) {$h_2$};
            \node[below=4pt] at (h3) {$h_3$};
            \node[below=4pt] at (h4) {$h_4$};
            \node[below=4pt] at (h5) {$h_5$};
            \node[above=4pt] at (x) {$x$};
            \node[above=4pt] at (y) {$y$};

            \node at (0,-0.85) {$R$};
        \end{tikzpicture}
        &
        \begin{tikzpicture}[
            scale=0.82,
            transform shape,
            baseline=(current bounding box.center)
        ]
            \useasboundingbox (-3.10,-1.35) rectangle (3.10,1.55);

            \node[v] (z1) at (-2.45,0) {};
            \node[v] (p) at (-1.25,0.82) {};
            \node[v] (r) at (-1.25,-0.82) {};

            \node[v] (q) at (1.25,0.82) {};
            \node[v] (s) at (1.25,-0.82) {};
            \node[v] (z2) at (2.45,0) {};

            \draw[blue, line width=0.9pt]
                (z1)--(p)--(r)--(z1)
                (z2)--(q)--(s)--(z2);

            \draw[
                blue,
                dashed,
                line width=0.9pt
            ]
                (p)--(q);

            \node[left=5pt] at (z1) {$z_1$};
            \node[above=4pt] at (p) {$p$};
            \node[below=4pt] at (r) {$h_2x$};

            \node[above=4pt] at (q) {$q$};
            \node[below=4pt] at (s) {$h_4y$};
            \node[right=5pt] at (z2) {$z_2$};

            \node at (0,-1.15) {$T$};
        \end{tikzpicture}
        \\[-0.2em]
        {\small underlying graph}
        &
        {\small corresponding stripe}
    \end{tabular}

    \caption{A representative stripe of type $\cZ_3$.}
    \label{fig:z3}
\end{figure}

The line-trigraph description again gives a simple structure after
either terminal is removed.

\begin{lemma}\label{lem:z3}
Let $S$ be the piece of a thickening of a stripe of type $\cZ_3$, and
let $A_1,A_2$ be its terminals. Then both $S-A_1$ and $S-A_2$ are
perfect.
\end{lemma}

\begin{proof}
Let $(X_e:e\in E(R))$ be the bags of the thickening. Since the marker
$h_1h_2$ is simplicial, $A_1$ is the union of the bags indexed by the
non-marker edges incident with $h_2$. Similarly, $A_2$ is the union
of the bags indexed by the non-marker edges incident with $h_4$.

After $A_1$ is deleted, every remaining base edge is incident with
$h_3$ or $h_4$. Let $D_3$ be the union of the bags indexed by the
edges incident with $h_3$, excluding $p$ and $q$, and let $D_4$ be
the union of the bags indexed by the edges incident with $h_4$,
excluding the marker $h_4h_5$.

Both $D_3$ and $D_4$ are cliques. Indeed, if the corresponding vertex
has degree at least three in $R$, the relevant base vertices are
strongly adjacent. If it has degree two, the corresponding set is
empty or consists of one bag. Since $S-A_1=D_3\cup D_4$, the graph
$S-A_1$ is cobipartite and therefore perfect. The same argument with
$h_2$ and $h_4$ interchanged shows that $S-A_2$ is perfect.
\end{proof}

\subsection{Type \texorpdfstring{$\cZ_4$}{Z4}}

Following Chudnovsky and Seymour~\cite[Section~7]{ChudnovskySeymour},
the base trigraph of a stripe of type $\cZ_4$ has vertex set
$\{a_0,a_1,a_2,b_0,b_1,b_2,b_3,c_1,c_2\}$. The sets
$\{a_0,a_1,a_2\}$, $\{b_0,b_1,b_2,b_3\}$,
$\{a_2,c_1,c_2\}$, and $\{a_1,b_1,c_2\}$ are strong cliques.
In addition, $b_2$ and $c_1$ are strongly adjacent, while
$b_2c_2$ and $b_3c_1$ are semiadjacent. Every other pair not
specified above is strongly antiadjacent. The markers are $a_0$ and
$b_0$.

A stripe of type $\cZ_4$ is shown in Figure~\ref{fig:z4}. Solid blue
lines represent strong adjacencies, and dashed blue lines represent
semiadjacencies.
\begin{figure}[H]
\centering
\tikzstyle{v}=[
circle,
draw,
fill=black,
inner sep=0pt,
minimum size=5pt
]

\begin{tikzpicture}[
    scale=0.88,
    transform shape,
    baseline=(current bounding box.center)
]
    \useasboundingbox (-4.85,-2.05) rectangle (4.85,2.05);

    % Left triangle
    \node[v] (a0) at (-4.00,0) {};
    \node[v] (a1) at (-2.70,1.20) {};
    \node[v] (a2) at (-2.70,-1.20) {};

    % Central vertices
    \node[v] (c2) at (-0.25,0.45) {};
    \node[v] (c1) at (-0.25,-0.55) {};

    % Right diamond
    \node[v] (b0) at (4.00,0) {};
    \node[v] (b1) at (2.70,1.20) {};
    \node[v] (b2) at (1.75,0) {};
    \node[v] (b3) at (2.70,-1.20) {};

    % Strong clique on the a-vertices
    \draw[blue, line width=0.9pt]
        (a0)--(a1)--(a2)--(a0);

    % Strong clique on the b-vertices
    \draw[blue, line width=0.9pt]
        (b1)--(b0)--(b3)--(b2)--(b1)
        (b1)--(b3)
        (b0)--(b2);

    % Strong clique on a2,c1,c2
    \draw[blue, line width=0.9pt]
        (a2)--(c1)--(c2)--(a2);

    % Strong clique on a1,b1,c2
    \draw[blue, line width=0.9pt]
        (a1)--(b1)--(c2)--(a1);

    % Additional strong adjacency
    \draw[blue, line width=0.9pt]
        (c1)--(b2);

    % Semiadjacent pairs
    \draw[blue, dashed, line width=0.9pt]
        (c2)--(b2)
        (c1)--(b3);

    % Vertex labels
    \node[left=7pt, align=center] at (a0)
        {$a_0$\\[-1pt]{\scriptsize marker}};
    \node[above left=2pt] at (a1) {$a_1$};
    \node[below left=2pt] at (a2) {$a_2$};

    \node[above=5pt] at (c2) {$c_2$};
    \node[below=5pt] at (c1) {$c_1$};

    \node[right=7pt, align=center] at (b0)
        {$b_0$\\[-1pt]{\scriptsize marker}};
    \node[above=5pt] at (b1) {$b_1$};
    \node[above left=3pt] at (b2) {$b_2$};
    \node[below=5pt] at (b3) {$b_3$};
\end{tikzpicture}

\caption{The stripe of type $\cZ_4$.}
\label{fig:z4}
\end{figure}
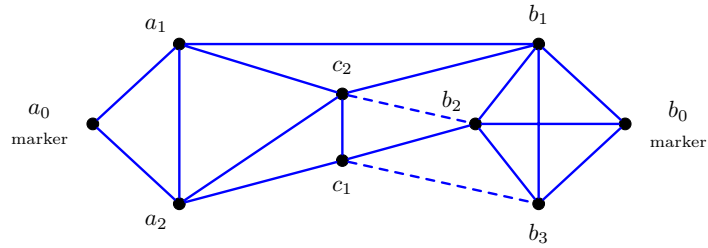
As in the preceding two types, deleting either terminal leaves a
cobipartite graph.

\begin{lemma}\label{lem:z4}
Let $S$ be the piece of a thickening of a stripe of type $\cZ_4$, and
let $A_1,A_2$ be its terminals. Then both $S-A_1$ and $S-A_2$ are
perfect.
\end{lemma}

\begin{proof}
Let $(X_v:v\in V(T))$ be the bags of the thickening. The terminals are
$A_1=X_{a_1}\cup X_{a_2}$ and
$A_2=X_{b_1}\cup X_{b_2}\cup X_{b_3}$.

After $A_1$ is deleted, the remaining vertices are covered by the two
cliques $A_2$ and $X_{c_1}\cup X_{c_2}$. Hence $S-A_1$ is
cobipartite and perfect. After $A_2$ is deleted, the remaining vertices
are covered by the cliques $A_1$ and $X_{c_1}\cup X_{c_2}$. Thus
$S-A_2$ is also cobipartite and perfect.
\end{proof}

\subsection{Type \texorpdfstring{$\cZ_5$}{Z5}}

Following Chudnovsky and Seymour~\cite[Section~7]{ChudnovskySeymour},
we only need the members of $\cZ_5$ with two markers. In this case
$v_7\notin X$. Since $v_{13}$ is adjacent to both $v_7$ and $v_8$, the
definition of a stripe forces $v_{13}\in X$. Thus such a stripe
contains the markers $v_7,v_8$, the vertices
$v_1,\ldots,v_6,v_9,v_{10}$, and possibly $v_{11}$ and $v_{12}$.
The vertices $v_1,\ldots,v_6$ induce the cycle
$v_1v_2\cdots v_6v_1$. The remaining adjacencies are given by
\begin{align*}
N_T(v_7)&=\{v_1,v_2\},\\
N_T(v_8)&=\{v_4,v_5\},\\
N_T(v_9)&=\{v_6,v_1,v_2,v_3,v_{10},v_{11},v_{12}\},\\
N_T(v_{10})&=\{v_3,v_4,v_5,v_6,v_9,v_{11},v_{12}\},\\
N_T(v_{11})&=\{v_1,v_3,v_4,v_6,v_9,v_{10}\},\\
N_T(v_{12})&=\{v_2,v_3,v_5,v_6,v_9,v_{10}\},
\end{align*}
where an absent vertex is omitted. The pair $v_7v_8$ is strongly
antiadjacent. Every listed adjacent pair is strongly adjacent, except
that $v_9v_{10}$ may be semiadjacent. Every unlisted pair is strongly
antiadjacent. The markers are $v_7$ and $v_8$.

Figure~\ref{fig:z5} shows a representative member in which
$v_{11}$ and $v_{12}$ are absent and $v_9v_{10}$ is semiadjacent.
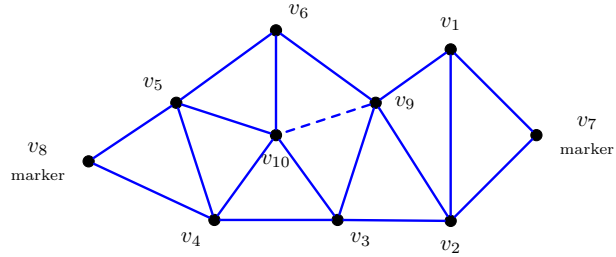
\begin{figure}[H]
\centering
\tikzstyle{v}=[
circle,
draw,
fill=black,
inner sep=0pt,
minimum size=5pt
]

\begin{tikzpicture}[
    scale=0.84,
    transform shape,
    baseline=(current bounding box.center)
]
    \useasboundingbox (-4.45,-2.10) rectangle (4.45,2.10);

    % Centre of the regular pentagon
    \node[v] (v10) at (0,0) {};

    % The regular pentagon centred at v10,
    % with v3v4 as its horizontal bottom edge
    \node[v] (v3) at (-54:1.65) {};
    \node[v] (v4) at (-126:1.65) {};
    \node[v] (v5) at (162:1.65) {};
    \node[v] (v6) at (90:1.65) {};
    \node[v] (v9) at (18:1.65) {};

    % Remaining vertices
    \node[v] (v1) at (2.75,1.35) {};
    \node[v] (v2) at (2.75,-1.35) {};

    % Markers
    \node[v] (v7) at (4.10,0) {};
    \node[v] (v8) at (-2.95,-0.4125) {};

    % The 6-cycle
    \draw[blue, line width=0.9pt]
        (v1)--(v2)--(v3)--(v4)--(v5)--(v6)--cycle;

    % Marker adjacencies
    \draw[blue, line width=0.9pt]
        (v7)--(v1)
        (v7)--(v2)
        (v8)--(v4)
        (v8)--(v5);

    % Adjacencies of v9
    \draw[blue, line width=0.9pt]
        (v9)--(v6)
        (v9)--(v1)
        (v9)--(v2)
        (v9)--(v3);

    % Adjacencies of v10
    \draw[blue, line width=0.9pt]
        (v10)--(v3)
        (v10)--(v4)
        (v10)--(v5)
        (v10)--(v6);

    % Possible semiadjacent pair
    \draw[blue, dashed, line width=0.9pt]
        (v9)--(v10);

    % Labels on the 6-cycle
    \node[above=5pt] at (v1) {$v_1$};
    \node[below=5pt] at (v2) {$v_2$};
    \node[below right=3pt] at (v3) {$v_3$};
    \node[below left=3pt] at (v4) {$v_4$};
    \node[above left=3pt] at (v5) {$v_5$};
    \node[above right=3pt] at (v6) {$v_6$};

    % Labels on the central vertices
    \node[right=5pt] at (v9) {$v_9$};
    \node[below=5pt] at (v10) {$v_{10}$};

    % Marker labels
    \node[right=7pt, align=center] at (v7)
        {$v_7$\\[-1pt]{\scriptsize marker}};

    \node[left=7pt, align=center] at (v8)
        {$v_8$\\[-1pt]{\scriptsize marker}};
\end{tikzpicture}

\caption{A representative stripe of type $\cZ_5$.}
\label{fig:z5}
\end{figure}
Unlike the first four types, deleting a terminal need not leave a
perfect graph. Instead, the two central bags provide a perfect set
that meets every maximal clique.

\begin{lemma}\label{lem:z5}
Let $S$ be the piece of a thickening of a stripe of type $\cZ_5$, and
let $A_1,A_2$ be its terminals. Then $A_1$ is anticomplete to $A_2$.
Moreover, there is a set
$C\subseteq V(S)\setminus(A_1\cup A_2)$ such that

\begin{enumerate}[label=\textup{$(\roman*)$}]
\item $C$ meets every maximal clique of $S$; and
\item both $S[C\cup A_1]$ and $S[C\cup A_2]$ are perfect.
\end{enumerate}
\end{lemma}

\begin{proof}
Let $(X_v:v\in V(T))$ be the bags of the thickening. The two terminals
are $A_1=X_{v_1}\cup X_{v_2}$ and
$A_2=X_{v_4}\cup X_{v_5}$. The definition of $\cZ_5$ shows that
$A_1$ is anticomplete to $A_2$.

Put $C=X_{v_9}\cup X_{v_{10}}$ and
$D=X_{v_3}\cup X_{v_6}\cup X_{v_{11}}\cup X_{v_{12}}$, where the
bags of absent vertices are omitted. The bag $X_{v_9}$ is complete
to $A_1\cup D$, while $X_{v_{10}}$ is complete to $A_2\cup D$.

Let $K$ be a clique of $S$ that misses $C$. Since $A_1$ is
anticomplete to $A_2$, the clique $K$ cannot meet both terminals. If
$K$ meets $A_1$, then every vertex of $X_{v_9}$ is complete to $K$.
If $K$ meets $A_2$, then every vertex of $X_{v_{10}}$ is complete to
$K$. If $K$ meets neither terminal, then $K\subseteq D$, and both
$X_{v_9}$ and $X_{v_{10}}$ are complete to $K$. In each case, $K$
is not maximal. Therefore $C$ meets every maximal clique of $S$.

The graph $S[C]$ is cobipartite, with clique partition
$(X_{v_9},X_{v_{10}})$, and is therefore perfect. The terminal $A_1$
is complete to $X_{v_9}$ and anticomplete to $X_{v_{10}}$. Hence
Lemma~\ref{lem:clique-attachment} implies that $S[C\cup A_1]$ is
perfect. Similarly, $A_2$ is complete to $X_{v_{10}}$ and
anticomplete to $X_{v_9}$, so $S[C\cup A_2]$ is perfect.
\end{proof}

\subsection{Unify pieces}
The preceding five lemmas can now be stated in the form required for
combining the pieces. An orientation of a piece specifies which
terminal will be included in the transversal and which one will be
avoided.

\begin{lemma}\label{lem:oriented-local-set}
Let $G$ be a graph, let $G[V_e]$ be a piece of type
$\cZ_1,\ldots,\cZ_5$, and let $A_e^r,A_e^s$ be its terminals.
Suppose that the edge indexing this piece is oriented from $r$ to $s$
and that $|A_e^r|<\omega(G)$. Then there is a set
$P_e\subseteq V_e$ such that

\begin{enumerate}[label=\textup{$(\roman*)$}]
\item $A_e^s\subseteq P_e$ and $P_e\cap A_e^r=\emptyset$;
\item $G[P_e]$ is perfect;
\item $P_e\cap K\neq\emptyset$ for every $K\in\Omega_G$ with
$K\subseteq V_e$.
\end{enumerate}
\end{lemma}

\begin{proof}
Suppose first that the piece has type
$\cZ_1,\cZ_2,\cZ_3$, or $\cZ_4$. Put
$P_e=V_e\setminus A_e^r$. Since the terminals are disjoint,
$A_e^s\subseteq P_e$. The graph $G[P_e]$ is perfect by the
corresponding one of Lemmas~\ref{lem:z1}, \ref{lem:z2}, \ref{lem:z3},
and~\ref{lem:z4}. If a maximum
clique $K$ of $G$ contained in $V_e$ missed $P_e$, then
$K\subseteq A_e^r$. Consequently,
$|K|\leq |A_e^r|<\omega(G)$, a contradiction.

Now suppose that the piece has type $\cZ_5$. Let $C_e$ be the set
given by Lemma~\ref{lem:z5}, and put $P_e=C_e\cup A_e^s$. Then
$P_e$ contains $A_e^s$, avoids $A_e^r$, and induces a perfect graph.
Every $K\in\Omega_G$ contained in $V_e$ is a maximal clique of
$G[V_e]$. Hence $K$ meets $C_e$ by Lemma~\ref{lem:z5}, and therefore
meets $P_e$.
\end{proof}

\section{Perfect transversals for compositions}\label{sec:composition}

We have obtained the required perfect set inside each non-spot piece. We now
choose compatible orientations and combine the local sets from
Lemma~\ref{lem:oriented-local-set} into one perfect
$\Omega_G$-transversal. The spot pieces require separate treatment. Since
their adjacencies form a line graph, we begin with an elementary lemma about
cuts in multigraphs.

A \emph{loopless multigraph} may have parallel edges but has no loops.
Its line graph $L(B)$ has vertex set $E(B)$, where two distinct
vertices are adjacent if the corresponding edges of $B$ have a common
end. A \emph{cut} of $B$ is the set of edges with one end in each part
of a partition of $V(B)$.

\begin{lemma}\label{lem:cut}
Let $B$ be a loopless multigraph with at least one edge. Then $B$ has
a cut $F$ such that

\begin{enumerate}[label=\textup{$(\roman*)$}]
\item every vertex of positive degree in $B$ is incident with an edge
of $F$;
\item $L(B)[F]$ is perfect;
\item $F$ is an $\Omega_{L(B)}$-transversal.
\end{enumerate}
\end{lemma}

\begin{proof}
Choose a cut $F$ of maximum size. For $v\in V(B)$, let $c(v)$ be the
number of edges of $F$ incident with $v$, and let $i(v)$ be the number
of edges outside $F$ incident with $v$, where parallel edges are
counted separately. Moving $v$ to the other side of the cut changes
its size by $i(v)-c(v)$. The maximality of $F$ therefore gives
$c(v)\geq i(v)$. If $v$ has positive degree, this inequality implies
$c(v)>0$. This proves \textup{(i)}.

Let $B_F$ be the spanning submultigraph of $B$ with edge set $F$.
Since $F$ is a cut, $B_F$ is bipartite. Moreover,
$L(B)[F]=L(B_F)$. By K\H{o}nig's line-coloring
theorem~\cite{LovaszPlummer}, the chromatic number of the line graph
of a bipartite multigraph equals its maximum degree.

Every pairwise incident family of edges in a bipartite multigraph has
a common end. Indeed, a pairwise incident family without a common end
must be supported on the three sides of a triangle, which cannot occur
in a bipartite multigraph. Thus the clique number of the line graph of
a bipartite multigraph also equals its maximum degree. The same
argument applies to every submultigraph of $B_F$, so $L(B_F)$ is
perfect. This proves \textup{(ii)}.

Let $Q$ be a maximum clique of $L(B)$. We first recall the form of
$Q$. A pairwise incident family of edges in a loopless multigraph
either has a common end or is supported on the three sides of a
triangle. To see this, choose an edge $ab$ in a family with no common
end. There is an edge $bc$ that does not contain $a$ and an edge $ad$
that does not contain $b$. These two edges must meet, so $c=d$. Every
other edge in the family is parallel to one of $ab$, $bc$, or $ca$.

Suppose first that all edges of $Q$ have a common end $v$. Since $Q$
is maximal, it contains every edge incident with $v$. By
\textup{(i)}, at least one of these edges belongs to $F$. Hence
$Q\cap F\neq\emptyset$.

It remains to consider the case in which $Q$ is supported on a
triangle with vertices $a,b,c$. Suppose that $Q\cap F=\emptyset$.
Since $Q$ is maximal, it contains every edge in the three bundles
between $a,b,c$. Let $x,y,z>0$ be the numbers of edges in the
$ab$-, $bc$-, and $ca$-bundles, respectively, and put
$t=x+y+z=|Q|=\omega(L(B))$.

All three vertices $a,b,c$ lie on the same side of the cut. At $a$,
we have $i(a)\geq x+z$ and $c(a)\geq i(a)$. Since the full star at
$a$ is a clique of $L(B)$, we also have $c(a)+i(a)\leq t$. It follows
that $2(x+z)\leq t$. Similarly, $2(x+y)\leq t$ and
$2(y+z)\leq t$. Adding these three inequalities gives $4t\leq3t$, a
contradiction. Therefore $Q\cap F\neq\emptyset$, proving
$(iii)$.
\end{proof}

We use the definition of a composition given in Section~\ref{sec:main-claw-free}.
Let $E_0$ be the set of spot edges of $R$, let
$E_1=E(R)\setminus E_0$, and put $B=(V(R),E_0)$. If
$E_0=\emptyset$, put $F=\emptyset$ and $D=\emptyset$. Otherwise,
apply Lemma~\ref{lem:cut} to $B$, let $F\subseteq E_0$ be the
resulting cut, and let $D$ be the set of vertices of $R$ incident
with an edge of $F$. By Lemma~\ref{lem:cut}\textup{(i)}, $D$ is
exactly the set of vertices of $R$ incident with at least one spot
edge.

The selected spot vertices are $\{p_f:f\in F\}$. The next lemma
orients the non-spot edges so that every remaining hub receives a
terminal from one of its incident pieces.

\begin{lemma}\label{lem:orientation}
Suppose that $G$ is connected, has no clique cutset, and is a
composition over $R$. The edges of $E_1$ can be oriented so that every
vertex of $R$ outside $D$ that is incident with an edge of $E_1$ has
an incoming edge. Moreover, if $e=rs$ is oriented from $r$ to $s$,
then $|A_e^r|<\omega(G)$.
\end{lemma}

\begin{proof}
Let $T$ be a component of the multigraph $(V(R),E_1)$ that contains
an edge. Suppose first that $T$ contains a cycle, where two parallel
edges may form a cycle of length two. Orient one such cycle cyclically.
Contract the cycle, choose a spanning tree of the resulting component,
and orient its edges away from the contracted cycle. Orient all
remaining edges arbitrarily. Every vertex of $T$ then has an incoming
edge.

Now suppose that $T$ is a tree. We claim that $V(T)\cap D\neq
\emptyset$. Suppose otherwise. Since $D$ contains every vertex
incident with a spot edge, no vertex of $T$ is incident with a spot
edge. Moreover, since $T$ is a component of $(V(R),E_1)$, no non-spot
edge joins a vertex of $T$ to a vertex outside $T$.

If $R$ had an edge outside $T$, the corresponding piece would be
anticomplete to all pieces indexed by the edges of $T$, contrary to
the connectedness of $G$. Thus every edge of $R$ belongs to $T$.
Since $R$ has at least two edges, $T$ has a vertex $r$ of degree at
least two.

For each component $U$ of $T-r$, let $rt$ be the edge joining $r$ to
$U$, and put
$W_U=(V_{rt}\setminus A_{rt}^r)\cup
\bigcup_{e\in E(T[U])}V_e$. The sets $W_U$ partition
$V(G)\setminus H_r$. Each $W_U$ is nonempty because it contains the
nonempty terminal $A_{rt}^t$. Moreover, sets corresponding to
different components of $T-r$ are anticomplete. Hence $H_r$ is a
clique cutset of $G$, a contradiction. Therefore
$V(T)\cap D\neq\emptyset$.

Choose $d\in V(T)\cap D$, root $T$ at $d$, and orient every edge away
from $d$. Every vertex other than $d$ has an incoming edge, while $d$
is incident with a selected spot edge.

Apply this construction to every component of $(V(R),E_1)$. Let
$e=rs$ be oriented from $r$ to $s$. The vertex $r$ is incident with
either a selected spot edge or an incoming non-spot edge different
from $e$. In either case, the hub $H_r$ contains a vertex outside
$A_e^r$. Since $H_r$ is a clique, we obtain
$|A_e^r|<|H_r|\leq\omega(G)$.
\end{proof}

We can now combine the selected spot vertices and the local sets from
the non-spot pieces.

\begin{theorem}\label{thm:composition}
Let $G$ be a connected graph with no clique cutset. If $G$ is a
composition over a loopless multigraph $R$ with at least two edges,
then $G$ has a perfect $\Omega_G$-transversal.
\end{theorem}

\begin{proof}
Orient the edges of $E_1$ as in Lemma~\ref{lem:orientation}. For every
edge $e=\overrightarrow{rs}$ of $E_1$, apply
Lemma~\ref{lem:oriented-local-set} and let $P_e$ be the resulting set.
Put $P=\{p_f:f\in F\}\cup\bigcup_{e\in E_1}P_e$.

We first record the precise location of a clique that meets more than
one piece.

\begin{claim}\label{claim:clique-location}
Let $K$ be a clique of $G$ that meets a non-spot piece $V_e$, where
$e=rs$, and at least one other piece. Then
\[
K\cap V_e\subseteq A_e^r\cup A_e^s.
\]
If $K\cap V_e$ is contained in one terminal, then $K$ is contained in
the corresponding hub. Otherwise, $K$ meets both terminals, and every
other piece met by $K$ is a spot piece indexed by an edge with ends
$r$ and $s$.
\end{claim} \vspace{-0.6em}

Every vertex of $V_e\setminus(A_e^r\cup A_e^s)$ has no neighbor in a
different piece, which proves the displayed containment. Suppose that
$K\cap V_e\subseteq A_e^r$. If $f\neq e$ and
$y\in K\cap V_f$, then the cross-edge rule, applied to any
$x\in K\cap V_e$, forces $f$ to be incident with $r$ and
$y\in A_f^r$. Hence every vertex of $K$ belongs to $H_r$. The same
argument applies with $s$ in place of $r$.

It remains to consider the case in which there are vertices
$x_r\in K\cap A_e^r$ and $x_s\in K\cap A_e^s$. Let
$f\neq e$ and $y\in K\cap V_f$. Since $y$ is adjacent to both
$x_r$ and $x_s$, the cross-edge rule gives
$y\in A_f^r\cap A_f^s$ and shows that the ends of $f$ are $r$ and
$s$. The terminals of a non-spot piece are disjoint. Therefore $f$ is
a spot edge, as required. This proves
Claim~\ref{claim:clique-location}.\hfill$\blacksquare$

\begin{claim}\label{claim:composition-transversal}
$P$ is an $\Omega_G$-transversal.
\end{claim} \vspace{-0.6em}

Let $K\in\Omega_G$. Suppose first that $K$ is contained in a
non-spot piece $V_e$. Then $K\cap P_e\neq\emptyset$ by
Lemma~\ref{lem:oriented-local-set}.

Suppose next that $K$ is contained in a hub $H_r$. Since $H_r$ is a
clique and $|K|=\omega(G)$, we have $K=H_r$. If $r\in D$, then $H_r$
contains a selected spot vertex $p_f$ with $f\in F$. If $r\notin D$,
then $r$ is incident with a non-spot edge. By
Lemma~\ref{lem:orientation}, some non-spot edge $e=t\to r$ is directed
into $r$. Since $A_e^r\subseteq P_e$, the clique $K$ meets $P_e$.

Now suppose that $K$ meets a non-spot piece $V_e$, meets another
piece, and is not contained in a hub. By
Claim~\ref{claim:clique-location}, the clique $K$ meets both terminals
of $V_e$. The set $P_e$ contains the terminal at the head of $e$, and
therefore $K\cap P_e\neq\emptyset$.

It remains to consider the case in which $K$ consists only of spot
vertices. The corresponding edges of $B$ are pairwise incident. If
they have a common end, then $K$ is contained in a hub, which has
already been considered. Otherwise, they are supported on the three
sides of a triangle. The spot vertices induce $L(B)$. Since $L(B)$ is
an induced subgraph of $G$ and $K$ is a maximum clique of $G$, the
clique $K$ is also maximum in $L(B)$. Lemma~\ref{lem:cut}\textup{(iii)}
now gives $K\cap\{p_f:f\in F\}\neq\emptyset$. This proves
Claim~\ref{claim:composition-transversal}.\hfill$\blacksquare$

\begin{claim}\label{claim:composition-perfect}
$G[P]$ is perfect.
\end{claim} \vspace{-0.6em}

The selected spot vertices induce $L(B)[F]$, which is perfect by
Lemma~\ref{lem:cut}\textup{(ii)}. List the edges of $E_1$ as
$e_1,\ldots,e_m$, and add the sets
$P_{e_1},\ldots,P_{e_m}$ in this order.

Suppose that $e_i=r\to s$. The set $P_{e_i}$ contains $A_{e_i}^s$
and avoids $A_{e_i}^r$. Every vertex of $P_{e_i}$ outside
$A_{e_i}^s$ has no neighbor in another piece. Let $Q_i$ be the set
of vertices already selected in the hub $H_s$. Then $Q_i$ is a clique,
and the only edges between $P_{e_i}$ and the vertices already added
are all the edges between $A_{e_i}^s$ and $Q_i$.

Since $G[P_{e_i}]$ is perfect,
Lemma~\ref{lem:clique-attachment} shows that adding $P_{e_i}$
preserves perfectness. Repeating this for
$i=1,\ldots,m$ proves that $G[P]$ is perfect. This proves
Claim~\ref{claim:composition-perfect}.\hfill$\blacksquare$

\medskip
Therefore $P$ is a perfect $\Omega_G$-transversal. This completes the
proof of Theorem~\ref{thm:composition}.
\end{proof}
\section{Proof of the main theorem}\label{sec:main-results}

The weight function is handled by replacing each vertex with a clique
whose size is its weight. This turns maximum-weight cliques into
ordinary maximum cliques and allows us to apply
Theorem~\ref{thm:structure}.

\begin{theorem}\label{thm:claw-weighted}
Every claw-free graph is perfectly weight divisible.
\end{theorem}

\begin{proof}
Suppose not, and choose a claw-free graph $G$ that is not perfectly
weight divisible with $|V(G)|$ minimum. Then every proper induced
subgraph of $G$ is perfectly weight divisible. Since $G$ is not
perfectly weight divisible, there are an induced subgraph $H$ of $G$
with at least one edge and a positive integral weight function $h_H$
on $V(H)$ such that $H$ has no $h_H$-perfect division. If $H$ were
proper, this would contradict the perfect weight divisibility of $H$.
Hence $H=G$. Writing $h=h_H$, we conclude that $G$ has no
$h$-perfect division. By Lemmas~\ref{lem:minimal}
and~\ref{lem:cutset}, $G$ is connected and has neither a simplicial vertex
nor a clique cutset.

For each $v\in V(G)$, replace $v$ with a clique $X_v$ of size $h(v)$.
For distinct $u,v\in V(G)$, $X_u$ is complete to $X_v$ if
$uv\in E(G)$, and is anticomplete otherwise. Denote the
resulting graph by $G^h$.

\begin{claim}\label{claim:blowup-structure}
$G^h$ is connected and claw-free and has neither a simplicial vertex
nor a clique cutset.
\end{claim} \vspace{-0.6em}

The connectedness of $G^h$ follows from that of $G$. An induced claw
in $G^h$ contains at most one vertex from each bag, and the indices of
its four bags would induce a claw in $G$. Hence $G^h$ is claw-free.

Suppose that some $x\in X_v$ is simplicial in $G^h$. If $u$ and $w$
are two neighbors of $v$ in $G$, then every vertex of $X_u\cup X_w$
is adjacent to $x$. Since $x$ is simplicial, $X_u$ is complete to
$X_w$, and hence $uw\in E(G)$. Thus $N_G(v)$ is a clique, so $v$ is
simplicial in $G$, a contradiction.

Suppose that $G^h$ has a clique cutset, and choose an
inclusion-minimal clique cutset $Q$. We first show that $Q$ is a union
of whole bags. Suppose that $Q$ contains some but not all vertices of
a bag $X_v$. Choose $x\in Q\cap X_v$ and $y\in X_v\setminus Q$.
The vertices $x$ and $y$ are true twins outside their bag: every
vertex outside $X_v$ is adjacent to either both or neither of them.
Moreover, $X_v$ is a clique. Hence every neighbor of $x$ outside $Q$
is adjacent to $y$ and belongs to the component of $G^h-Q$ containing
$y$. Therefore, adding $x$ back cannot join two
components of $G^h-Q$. It follows that $Q\setminus\{x\}$ is still a
clique cutset, contrary to the minimality of $Q$.

Thus $Q=\bigcup_{v\in D}X_v$ for some $D\subseteq V(G)$. Since $Q$
is a clique, $D$ is a clique of $G$. Moreover, connectivity among the
bags outside $Q$ is exactly connectivity in $G-D$. Hence $G-D$ is
disconnected, so $D$ is a clique cutset of $G$, a contradiction.
This proves Claim~\ref{claim:blowup-structure}.\hfill$\blacksquare$

\begin{claim}\label{claim:blowup-transversal}
$G^h$ has no perfect $\Omega_{G^h}$-transversal.
\end{claim} \vspace{-0.6em}

Suppose that $T$ is a perfect $\Omega_{G^h}$-transversal, and put
$P=\{v\in V(G):T\cap X_v\neq\emptyset\}$. Choosing one vertex from
each nonempty set $T\cap X_v$ gives an induced copy of $G[P]$ inside
$G^h[T]$. Since $G^h[T]$ is perfect, $G[P]$ is perfect.

The support of every clique of $G^h$ is a clique of $G$, and every
clique can be extended to contain all vertices of each bag that it
meets. Consequently, $\omega(G^h)=\omega_h(G)$. If
$K\in\Omega_{G,h}$, then $\bigcup_{v\in K}X_v$ is a maximum clique of
$G^h$. Since $T$ meets this clique, $P\cap K\neq\emptyset$.
Therefore $P$ is a perfect $\Omega_{G,h}$-transversal. By
Observation~\ref{obs:weighted-transversal}, $G$ has an $h$-perfect
division, a contradiction. This proves
Claim~\ref{claim:blowup-transversal}.\hfill$\blacksquare$

\medskip
Apply Theorem~\ref{thm:structure} to $G^h$. In outcome
\textup{(i)}, Lemma~\ref{lem:three-cliques} gives a perfect
$\Omega_{G^h}$-transversal. In outcome \textup{(ii)}, use
Lemma~\ref{lem:icosahedral} when the graph belongs to $\cS_1$,
Lemma~\ref{lem:circular-antineighborhood} when it belongs to $\cS_3$,
and Lemma~\ref{lem:antiprismatic-antineighborhood} when it belongs to
$\cS_7$. In outcome \textup{(iii)},
Theorem~\ref{thm:composition} gives a perfect
$\Omega_{G^h}$-transversal. Every outcome contradicts
Claim~\ref{claim:blowup-transversal}. This completes the proof of
Theorem~\ref{thm:claw-weighted}.
\end{proof}

\begin{proof}[Proof of Theorem~\ref{thm:main}]
Suppose that some fork-free graph is not perfectly weight divisible, and
choose a minimal non-perfectly weight divisible induced subgraph $G$ of such
a graph. The graph $G$ is fork-free, and Theorem~\ref{thm:xu-zhuang} implies
that $G$ is claw-free. This contradicts
Theorem~\ref{thm:claw-weighted}. This completes the proof of
Theorem~\ref{thm:main}.
\end{proof}

\section{Concluding remarks}\label{concluding}

Theorem~\ref{thm:main} deals with the fork, which is obtained from a claw
by subdividing one edge once. It is natural to ask whether the same
conclusion holds for other trees with similar structures. For positive integers $a,b,c$, let
$S_{a,b,c}$ be the tree formed by three paths of lengths $a,b,c$ that have
one common end and are otherwise vertex-disjoint. Thus the fork is
$S_{1,1,2}$. We also write $E=S_{1,2,2}$; this is the graph obtained from a
claw by subdividing two different edges once. The four trees used below
are shown in Figure~\ref{fig:nearby-trees}.

\begin{figure}[H]
\centering

\begin{minipage}[c]{0.22\textwidth}
\centering
\begin{tikzpicture}[scale=0.85]
\useasboundingbox (-1.25,-1.35) rectangle (1.25,1.35);
\node[vtx] (c) at (0,0) {};
\node[vtx] (a) at (0,1.00) {};
\node[vtx] (b) at (1.00,0) {};
\node[vtx] (d) at (0,-1.00) {};
\node[vtx] (e) at (-1.00,0) {};
\draw[ed] (c)--(a) (c)--(b) (c)--(d) (c)--(e);
\end{tikzpicture}

\smallskip
{\small\emph{$K_{1,4}$}}
\end{minipage}
\hfill
\begin{minipage}[c]{0.24\textwidth}
\centering
\begin{tikzpicture}[scale=0.85]
\useasboundingbox (-1.15,-1.35) rectangle (2.35,1.35);
\node[vtx] (c) at (0,0) {};
\node[vtx] (p) at (-0.85,0.80) {};
\node[vtx] (q) at (-0.85,-0.80) {};
\node[vtx] (a) at (0.75,0) {};
\node[vtx] (b) at (1.50,0) {};
\node[vtx] (d) at (2.25,0) {};
\draw[ed] (c)--(p) (c)--(q) (c)--(a)--(b)--(d);
\end{tikzpicture}

\smallskip
{\small\emph{$S_{1,1,3}$}}
\end{minipage}
\hfill
\begin{minipage}[c]{0.22\textwidth}
\centering
\begin{tikzpicture}[scale=0.85]
\useasboundingbox (-1.05,-1.35) rectangle (1.05,1.35);

\node[vtx] (v1) at (-0.45,0.95) {};
\node[vtx] (v2) at (-0.45,0) {};
\node[vtx] (v3) at (-0.45,-0.95) {};

\node[vtx] (u1) at (0.55,0.95) {};
\node[vtx] (u2) at (0.55,0) {};
\node[vtx] (u3) at (0.55,-0.95) {};

\draw[ed]
    (v1)--(v2)--(v3)
    (v1)--(u1)
    (v2)--(u2)
    (v3)--(u3);
\end{tikzpicture}

\smallskip
{\small\emph{$E=S_{1,2,2}$}}
\end{minipage}
\hfill
\begin{minipage}[c]{0.22\textwidth}
\centering
\begin{tikzpicture}[scale=0.80]
\useasboundingbox (-1.45,-1.35) rectangle (1.45,1.45);
\node[vtx] (c) at (0,0) {};
\foreach \i/\ang in {1/90,2/210,3/330}{
    \node[vtx] (u\i) at (\ang:0.75) {};
    \node[vtx] (z\i) at (\ang:1.45) {};
    \draw[ed] (c)--(u\i)--(z\i);
}
\end{tikzpicture}

\smallskip
{\small\emph{$S_{2,2,2}$}}
\end{minipage}

\caption{Four trees related to the fork.}
\label{fig:nearby-trees}
\end{figure}
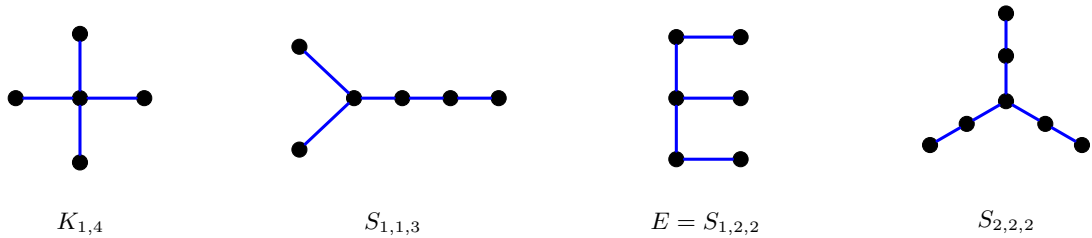

We recall some standard notation. An \emph{independent
set} of a graph is a set of pairwise nonadjacent vertices. The
independence number $\alpha(G)$ is the maximum size of an independent
set of $G$, and $\Delta(G)$ is the maximum degree of $G$. For a positive
integer $k$, let $P_k$ denote the path on $k$ vertices, and for
$k\geq3$, let $C_k$ denote the cycle on $k$ vertices. For integers
$s,t\geq2$, the Ramsey number $r(s,t)$ is the least integer $n$ such
that every graph on $n$ vertices contains a clique of size $s$ or an
independent set of size $t$.
We next show that forbidding any one of $K_{1,4}$, $S_{1,1,3}$, and
$S_{2,2,2}$ is not enough to ensure perfect divisibility. We first recall
two results that will be used below.

\begin{theorem}[Mattheus--Verstra\"ete~\cite{MattheusVerstraete2024}]
\label{thm:ramsey-4k}
There is a constant $c>0$ such that, for every sufficiently large
integer $k$, we have
$$
r(4,k)\geq \frac{ck^3}{\log^4 k}.
$$
\end{theorem}

Let $M$ be the Mycielski--Gr\"otzsch graph with vertex set
$\{w\}\cup\{x_i,y_i:i\in\mathbb Z_5\}$ and edges
$x_ix_{i+1}$, $wy_i$, $y_ix_{i-1}$, and $y_ix_{i+1}$ for
$i\in\mathbb Z_5$.(See Figure~\ref{fig:counterexample-graphs})

\begin{figure}[H]
\centering
\begin{minipage}[c]{0.39\textwidth}
\centering
\begin{tikzpicture}[scale=0.90]
  \foreach \i in {0,...,4}{
    \pgfmathsetmacro{\ang}{90-72*\i}
    \coordinate (x\i) at (\ang:2.25);
    \coordinate (y\i) at (\ang:1.10);
  }
  \coordinate (ww) at (0,0);
  \foreach \i in {0,...,4}{
    \pgfmathtruncatemacro{\j}{mod(\i+1,5)}
    \pgfmathtruncatemacro{\jm}{mod(\i+4,5)}
    \pgfmathtruncatemacro{\jp}{mod(\i+1,5)}
    \draw[ed] (x\i)--(x\j);
    \draw[blue,line width=0.85pt] (ww)--(y\i);
    \draw[blue,line width=0.85pt] (y\i)--(x\jm);
    \draw[blue,line width=0.85pt] (y\i)--(x\jp);
  }
  \node[vtx] at (ww) {};
  \node[font=\scriptsize,below=2pt] at (ww) {$w$};
  \foreach \i in {0,...,4}{
    \pgfmathsetmacro{\ang}{90-72*\i}
    \node[vtx] at (x\i) {};
    \node[vtx] at (y\i) {};
    \node[font=\scriptsize] at (\ang:2.62) {$x_{\i}$};
    \node[font=\scriptsize] at (\ang:1.43) {$y_{\i}$};
  }
\end{tikzpicture}

\smallskip
% {\small\emph{The Mycielski--Gr\"otzsch graph $M$.}}
\end{minipage}

\caption{The Mycielski--Gr\"otzsch graph $M$.}
\label{fig:counterexample-graphs}
\end{figure}
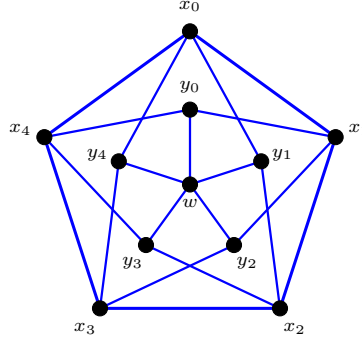

\begin{lemma}[Hu--Xiong~\cite{HuXiong2025}]
\label{lem:grotzsch-p6}
$M$ is not perfectly divisible.
\end{lemma}

\begin{proposition}\label{prop:three-counterexamples}
The following statements hold.
\begin{enumerate}[label=\textup{$(\roman*)$}]
\item There is a $K_{1,4}$-free graph that is not perfectly divisible.
\item There is an $S_{1,1,3}$-free graph that is not perfectly divisible.
\item There is an $S_{2,2,2}$-free graph that is not perfectly divisible.
\end{enumerate}
\end{proposition}

\begin{proof}
We first prove $(i)$. By
Theorem~\ref{thm:ramsey-4k}, we may choose $k$ sufficiently large that
$r(4,k)-1>3\binom{k}{2}$. By the definition of $r(4,k)$, there is a
graph $H$ on $r(4,k)-1$ vertices containing neither a $K_4$ nor a independent set of size $k$. Thus $\omega(H)\leq3$ and $\alpha(H)<k$.
Let $G=\overline H$. Then $\alpha(G)=\omega(H)\leq3$ and
$\omega(G)=\alpha(H)<k$. Since the four leaves of an induced
$K_{1,4}$ form a independent set, $G$ is $K_{1,4}$-free. Moreover,
$$
\chi(G)\geq\frac{|V(G)|}{\alpha(G)}
>\binom{k}{2}
\geq\binom{\omega(G)+1}{2}.
$$
This proves
$(i)$.

We next prove $(ii)$ and $(iii)$. By Lemma~\ref{lem:grotzsch-p6}, $M$ is not perfectly
divisible.
It remains to check the forbidden induced subgraphs. All indices are
taken modulo $5$. From the definition of $M$, we have
$N(w)=\{y_0,y_1,y_2,y_3,y_4\}$,
$N(x_i)=\{x_{i-1},x_{i+1},y_{i-1},y_{i+1}\}$, and
$N(y_i)=\{w,x_{i-1},x_{i+1}\}$. A direct check of the three possible
types of branch vertex, namely $w$, $x_i$, and $y_i$, shows that no
induced copy of $S_{1,1,3}$ or $S_{2,2,2}$ can occur in $M$. Indeed,
in each case either one of the required vertices does not exist or two
vertices that should be nonadjacent are adjacent. Therefore $M$ is both
$S_{1,1,3}$-free and $S_{2,2,2}$-free. This proves
$(ii)$ and $(iii)$.
\end{proof}

The examples above do not settle the case $E=S_{1,2,2}$.  This leads us to the following conjecture.

\begin{conjecture}\label{conj:E-free}
Every $E$-free graph is perfectly divisible.
\end{conjecture}

\section*{Acknowledgements}

This work was supported by the National Key R\&D Program of China
(No.~2022YFA1006400) and the National Natural Science Foundation of China
(No.~12571376).

\section*{Declaration}
\noindent\textbf{Conflict of interest.}
The authors declare that they have no known competing financial interests
or personal relationships that could have influenced the work reported in
this paper.

\noindent\textbf{Data availability.}
Data sharing is not applicable because no datasets were generated or
analyzed in this study.

\end{document}